\documentclass{amsart}

\usepackage{uconnreu} 
\usetikzlibrary{lindenmayersystems}

\DeclareMathOperator{\vol}{vol}

\title{Geodesic Interpolation on Sierpinski Gaskets}

\author[C. Davis]{Caitlin M. Davis}
\curraddr[C. Davis]{Department of Mathematical Sciences, Lewis \& Clark College, Portland, OR}
\email{caitlindavis@lclark.edu}

\author[L. LeGare]{Laura A. LeGare}
\curraddr[L. LeGare]{Department of Mathematics, University of Notre Dame, Notre Dame, IN}
\email{llegare@nd.edu}

\author[C. McCartan]{Cory W. McCartan}
\curraddr[C. McCartan]{Department of Statistics, Harvard University, Cambridge, MA}
\email{cmccartan@g.harvard.edu}

\author[L.~G.~Rogers]{Luke~G. Rogers}
\curraddr[L.~G.~Rogers]{Department of Mathematics, University of Connecticut, Storrs, CT}
\email{luke.rogers@uconn.edu}

\begin{document}

\begin{abstract}
We study the analogue of a convex interpolant of two sets on Sierpinski gaskets
and an associated notion of measure transport.  The structure of a natural
family of interpolating measures is described and an interpolation
inequality is established.  A key tool is a good description of geodesics on
these gaskets, some results on which have previously appeared in the
literature~\cites{strichartz1999,saltan2016,saltan2018,GuYeXi}.
\end{abstract}


\maketitle

The notion of a convex interpolant $(1-t)A+tB=\{(1-t)a+tb:a\in A, b\in B\}$ for
sets $A,B\subset\mathbb{R}^n$ and $t\in[0,1]$, and the Brunn-Minkowski
inequality $|(1-t)A+tB|^{1/n}\geq (1-t)|A|^{1/n}+t|B|^{1/n}$ for the $n$-volume
$|\cdot|$, have long had a central role in convex geometry.  More recently, the
class of functional inequalities that includes the Brunn-Minkowski inequality
has been used to dramatically extend notions of curvature to more general
settings, and a rich theory has developed around these
advances~\cites{Gardner,Barthe}.

The study of functional inequalities in the setting of fractal metric-measure
spaces is considerably less developed.  One area in which there has been a great
deal of work is in relating  the variation with $\epsilon>0$ of the volume of an
$\epsilon$-neighborhood of a set to the analytic and geometric properties of the
set.  For Euclidean sets with sufficiently smooth boundaries, such results can
be obtained using the Steiner formula and inequalities of Brunn-Minkowski type.
In the case of certain fractal sets and sets with fractal boundary in Euclidean
space, one achievement of the theory developed by Lapidus and collaborators is a
characterization of the volume of $\epsilon$-neighborhoods using complex
dimensions, which in turn are connected to analytic structure on the set through
the zeta function of its Laplacian~\cite{Lapidusbook}.     Functional
inequalities classically associated with curvature are also beginning to be
considered in fractal analytic settings~\cites{BaudoinKelleher,ABCRST3}.

A feature of the preceding work is that it does not generally use convex
interpolation.  Indeed, we are not aware of previous work involving convex
interpolation on fractal sets.  The purpose of this paper is to consider the
elementary notion of convex interpolant in the setting of one well-studied class
of fractals, the Sierpinski gaskets $S_n$ defined on regular $n$-simplices in
$\mathbb{R}^n$.  When endowed with the Euclidean metric restricted to the set,
these examples are geodesic spaces.  Following~\cite{cordero2001} we can
therefore define a convex interpolant $\tilde{Z}_t(A,B)$ which generalizes the Euclidean
notion of $(1-t)A+tB$ by setting
\begin{equation}
	\begin{split}
	\tilde{Z}_t(a,b)&=\{x: d(a,x)=td(a,b) \text{ and } d(x,b)=(1-t)d(a,b)\},\\
	\tilde{Z}_t(A,B)&= \{ Z_t(a,b):a\in A,b\in B\}.
	\end{split}
\end{equation}
Our goal is to study some basic properties of this interpolating set and the
naturally related notion of an interpolating measure on the sets $S_n$.

The study of $\tilde{Z}_t(A,B)$ requires that we have a good understanding of geodesics
in the Sierpinski gasket $S_n$.  These have been studied, for example
in~\cite{strichartz1999} and more recently~\cites{saltan2016,saltan2018,GuYeXi},
but we make some explicit constructions and reprove some fundamental theorems by
methods connected to the barycentric projection in~\cite{strichartz1999} because
they are essential in our treatment of $\tilde{Z}_t$.  The proofs are also, in our view,
simpler than some of those in~\cites{saltan2016,saltan2018,GuYeXi}.  These
results, including some which are new, are in Section~\ref{sec:geodesics}.  Our
study of interpolation occupies
Sections~\ref{sec:interp}--\ref{sec:interp-general}; we first replace  $\tilde{Z}_t(A,B)$ with a slightly simpler but essentially equivalent set $Z_t(A,B)$, then we deal with this set in the case where $A$ is a cell and $B$ a point and when $A$ and $B$ are
disjoint cells.  It is easy to determine that no direct analogue of the
Brunn-Minkowski inequality can hold, but we conclude with one possible
interpolation inequality in Section~\ref{sec:inequality}.

We emphasize to the reader that this study is intentionally limited in scope.
There are so many inequalities and applications of inequalities in this area
that it is not practical to attempt an exhaustive treatment, even when we limit
ourselves to such a  simple class of examples.  Moreover, the naturality of
convex interpolation in our setting is not discussed, and in particular we do
not consider whether the interpolation of measures is an optimizer of a
transport problem.  No doubt each reader will notice problems they think should
perhaps have been considered, and we hope they will be inspired to do so
themselves.


\section{Preliminaries}
\label{sec:prelim}

The \textit{Sierpinski $n$-gasket} $S_n\subseteq\R^n$ is the unique nonempty
compact attractor of the iterated function system (IFS) \[
    F_i : \mathbb{R}^n \to \mathbb{R}^n, \quad
    F_i(x) = \frac{1}{2}(x + q_i),
\] where $i \in \{0,1,\ldots,n\}$, and where $\{q_0,q_1,\ldots,q_n\}$ are the
vertices of an $n$-simplex with sides of unit length.  We begin with some
essential definitions and properties of $S_n$, mostly following the conventions
of Strichartz and coauthors in~\cites{azzam2008,strichartz}.

An \textit{$m$-level cell} of the Sierpinski $n$-gasket is a set of the form
$F_{w_1}\circ F_{w_2} \circ\cdots\circ F_{w_m}(S_n)$. We call the sequence of
letters $w = w_1w_2\cdots w_m$ from the alphabet $\{0,\dotsc,n\}$ a
\textit{finite address} of length $|w|=m$. We identify finite addresses with
cells and use the notation $\br{w}\dfeq F_w(S_n)$. An $(m+k)$-level cell
contained in a given $m$-level cell is called a \textit{level $k$ subcell} of
the $m$-level cell. A given $m$-level cell $\br{w}$ has $(n+1)$ 1-level
subcells, or \textit{maximal subcells}: $\br{w0},\br{w1},\ldots,\br{wn}$. Since
$S_n$ is defined on an $n$-simplex with unit length sides, the side length of an
$m$-level cell is $\qty(\half)^m$.

A strictly descending chain of cells
$S_n\supseteq\br{w_1}\supseteq\br{w_1w_2}\supseteq\cdots$ intersects to a point
with \textit{address} $w_1w_2\cdots$. As with cells, we identify infinite
addresses with points by writing $\br{w_1w_2\cdots}$.  We use an overline to
denote repeating characters in an address, so $20111\cdots=20\bar 1$. In this
notation, we have $q_j=\bar j$ for $j=0,\dotsc,n$.

\subsection{Vertices, address equivalence, and barycentric coordinates}\label{subsec:vertices-address-barycentric}

The \textit{boundary points} of a cell $\br{w}$ are the vertices
$\br{w\bar0},\br{w\bar1},\ldots,\br{w\bar n}$.  The \textit{$m$-level vertex
set} of the $n$-simplex, denoted $V_n^m$, is defined recursively by \[
  V_n^0=\{q_0,q_1, \ldots, q_n\} \qq{and} V_n^m=\bigcup_{i=0}^{n} F_i(V_n^{m-1}). 
\] 
The set of all vertices of $S_n$, denoted $V_n^*$, is defined as
$\bigcup_{m=0}^\infty V_n^m$;  this set is dense in $S_n$.

\begin{figure}[ht]
  \centering
  \begin{tikzpicture}[x=2in,y=\sqTHREE in]
  \PrepSierpinski
  \Sierpinski{7}
  \filldraw (1/2,1) circle (1.5pt) node[align=center,above] {$\br{\bar 0}$};
  \filldraw (0,0) circle (1.5pt) node[below left] {$\br{\bar 1}$};
  \filldraw (1,0) circle (1.5pt) node[below right] {$\br{\bar 2}$};
  
  \filldraw (1/2,0) circle (1.5pt) node[align=center,below] {$\br{2\bar1}=\br{1\bar2}$};
  \filldraw (5/8,3/4) circle (1.5pt) node[above right] {$\br{00\bar2}=\br{02\bar0}$};
  \filldraw (7/32,5/16) circle (1.5pt) node[above left,xshift=-4pt] {$\br{1012\bar0}=\br{1010\bar2}$};
  \end{tikzpicture}
  \caption{$S_2$, with several points labelled.} \label{fig:SG1}
\end{figure}
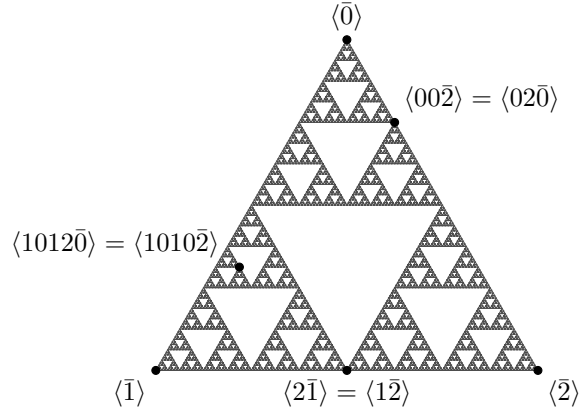

Every vertex sits at the intersection of two neighboring cells, and consequently
can be described by exactly two addresses, which are readily seen to have the
form $wj\bar i$ and $wi\bar j$, where $\br{wi}$ and $\br{wj}$ are the
intersecting cells. Figure \ref{fig:SG1} demonstrates this property in $S_2$ and
illustrates the addressing scheme. Each point in $S_n\setminus V_n^*$ has a
unique address.

In addition to point addresses, we make considerable use of the barycentric
coordinate system on $S_n$.  Recall that the convex hull of
$\{q_0,q_1,\ldots,q_n\}$, which contains $S_n$, consists of the points
\begin{equation}\label{eqn:barycoords}  x=c_0q_0+c_1q_1+ \cdots +c_nq_n,
\end{equation} in which each $c_j\geq0$ and $c_0+c_1+\cdots +c_n=1$.  The $c_j$
are called  the \textit{barycentric coordinates} of $x$, and we denote the
$i$\textsuperscript{th} barycentric coordinate $c_i$ of $x$ by $[x]_i$.

It is useful to consider the dyadic expansions $[x]_i=\sum_{j=1}^\infty c_i^j
2^{-j}$ of the barycentric coordinates, because of the following easy result
that is well known~\cite{Edgar}*{page 10} and will be used frequently throughout
the present work.

\begin{lemma}\label{lem:dyadbary}
A point $x$ is in $S_n$ if and only if there is a dyadic expansion of its
barycentric coordinates with the property that for each $j$ there is a
unique $i\in\{0,\dotsc,n\}$ so that $c_i^j=1$.  In fact,
$x=\br{w_1w_2\cdots}$  if and only if $c_i^j=1$ precisely when $w_j=i$.
\end{lemma}
\begin{proof}
Observe that points in the $1$-cell $\br{i}$ have $c_i^1=1$ and all other
$c_k^1=0$.  The result then follows by self-similarity and induction.
\end{proof}

\begin{remark}
Points in $V_n^*$ are those
for which each $c_i$ is a dyadic rational, and the two addresses for a vertex
correspond to the two (nonterminating) binary representations of the vertex.
For example, the vertex $\br{1\bar0}=\br{0\bar1}=(\frac{1}{2}, \frac{1}{2},0)$ 
in $S_2$ can be expressed in binary as either $(0.1, 0.0\bar{1}, 0)$ or
$(0.0\bar{1}, 0.1, 0)$; it may also be represented as $(0.1,0.1,0)$, but
this  latter fails to satisfy the condition that exactly one of the
$c_j^k=1$ for a given value of $k$.
\end{remark}

\subsection{Self-similar measure}\label{subsec:measures}

A natural class of measures on $S_n$ are the \textit{self-similar measures}.  A
measure $\mu_n$ of this type is a probability measure determined uniquely from a
set of weights $\{\mu_n^i\}_{i=0}^n$, where each $\mu_n^i>0$ and
$\sum_i\mu_n^i=1$, by the requirement that  for any measurable $X\subseteq S_n$
one has the self-similar identity
\[   \mu_n(X) = \sum_i\mu_n^i\mu_n(F_i^{-1}(X)).\]
The existence and uniqueness of such measures is due to Hutchinson~\cite{hutchinson1981}.

The \textit{standard measure}  on $S_n$ is self-similar with weights
$\mu_n^i=\onpo$ for all $i \in \{0,1,\ldots,n\}$.  The standard measure of any 
$m$-level cell is $(\onpo)^m$.

When we study interpolation in $S_n$, the interpolant sets will be determined by projections of 
the original sets.  The measure on interpolant sets will be a projection of the original 
self-similar measure, and will therefore be self-similar itself. This is a
consequence of the following results, which are well-known when each $F_i$ is a
homothety, as it is here.

\begin{lemma} \label{prop:general-IFS}
    Let $q_0,q_1, \ldots, q_n \in \R^m$, and consider the IFS $\{F_i:\R^m\to
    \R^m\}$ with $F_i(x)=\frac{1}{2}(x+q_i)$. Fix $\vec{v}$ a unit vector and 
    define $\phi:\R^m\to\R$ by $x\mapsto \langle x,v\rangle$. Then, defining 
    $\tilde F_i\dfeq\phi\circ F_i\circ\phi^{-1}$, we have $\tilde F_i(x)=\half
    (x+\phi(q_i))$ for $x\in\R$.
\end{lemma}
\begin{proof}
The map $\phi$ is linear, so if $y$ satisfies $\phi(y)=x$, then 
\begin{equation*}
  \phi\circ F_i (y)
  =\phi\qty(\frac{1}{2}(y + q_i)) 
  = \frac{1}{2}(\phi(y) + \phi(q_i)) 
  = \frac{1}{2}(x\vec{v} + \phi(q_i)).\qedhere
\end{equation*}
\end{proof}

\begin{prop}\label{prop:general-pushforward}
In the setting of Lemma~\ref{prop:general-IFS}, let $K$ denote the attractor of
    the IFS and $\mu$ be the self-similar measure on $K$ with weights $\mu^i$.
    Then the pushforward measure $\phi_*\mu$ satisfies the self-similar identity 
\[
  \phi_*\mu(X) = \sum_i \mu^i  \phi_*\mu(\tilde F_i^{-1}(X))
\]
for measurable $X\subseteq \phi(K)$.
\end{prop}
\begin{proof}
By the definition of the pushforward measure and the self-similarity of $\mu$, 
\begin{equation*}
\phi_*\mu(X) = \mu(\phi^{-1}(X)) 
= \sum_i \mu^i \mu(F_i^{-1}(\phi^{-1}(X))).
\end{equation*}
However $F_i^{-1}\circ \phi^{-1}(X)=\phi^{-1}\circ \tilde{F}_i^{-1}(X)$, because, using Lemma~\ref{prop:general-IFS},
\[ \phi\circ F_i(y)=\phi(\frac12(y+q_i))=\frac12(\phi(y)+\phi(q_i))= \tilde{F}_i\circ\phi(y).\]
Thus
\begin{equation*}
  \phi_*\mu(X) 
  = \sum_i \mu^i\, \mu(\phi^{-1}\circ\tilde F_i^{-1}(X)) 
  = \sum_i \mu^i\, \phi_*\mu(\tilde F_i^{-1}(X)). \qedhere
\end{equation*}
\end{proof}


\section{Geodesics}
\label{sec:geodesics}

Our goal in this section is to relate barycentric coordinates to distance, and to characterize nonuniqueness of geodesics in $S_n$.
We begin by considering geodesics from a point to a boundary point, and then generalize to 
arbitrary points in $S_n$. We prove that there exist at most five distinct geodesics between any two points in $S_2$, and at most eight geodesics between any two points in $S_n$, $n\geq3$.

Let $x,y\in S_n$. A \textit{path} from $x$ to $y$ in $S_n$ is a continuous function $\gamma:[0,1]\to S_n$ 
such that $\gamma(0)=x$ and $\gamma(1)=y$. We say that $\gamma$ \textit{passes through} a point $z$ if 
for some $t\in(0,1)$ we have $\gamma(t)=z$.  The length of a path $\gamma$, given by $\H^1(\gamma([0,1]))$, 
is denoted $|\gamma|$; a priori it may be infinite, but we will only be interested in finite paths.  To avoid the usual problem of distinguishing $\gamma$ from its image we will always assume $\gamma$ is parametrized at constant speed, so $\Bigl|\frac{\partial \H^1(\gamma([0,t]))}{\partial t}\Bigr|=|\gamma|$ for a.e.-$t$, unless some other parametrization is specified.

It is easy to see that there is always a finite path between any $x$ and $y$.  We then define the \textit{intrinsic metric}, $d:S_n\times S_n\to\R$  by 
\[	d(x, y) = \inf\{|\gamma|:\gamma\text{ a path from $x$ to $y$}\}.
\] 
This was previously investigated for the case of $S_2$ in~\cite[Section~8]{strichartz1999}, and later in~\cites{cristea2013, saltan2018}. In particular, the question of the existence of minimizing paths, or \textit{geodesics} has been considered in this case.  Strichartz~\cite{strichartz1999} used barycentric coordinates to give a simple construction of geodesics; we follow his method, but correct his statement that the  maximum number of geodesics between an arbitrary pair of points is four rather than five.  Saltan et al.~\cite{saltan2018}   present a formula for $d$ in terms of address representations of $x$ and $y$, and obtain the correct value for the maximum number of geodesics between $x$ and $y$, but their approach is somewhat complicated and is only done for $n=2$.  This result was generalized in~\cite{GuYeXi} to the case $n\geq3$.  The rest of this section presents an alternative approach to these results for $S_n$, $n\geq3$ and fixes notation that will be needed in our study of geodesic interpolation.

We begin by studying the problem of connecting a boundary point $x$ of a cell to a point $y$ in that cell by a geodesic.  The following lemma proves that if such a geodesic exists it must lie in the cell.

\begin{lemma}\label{lem:geostaysincell}
Two boundary points of a cell $\br{w}$ are joined by a unique geodesic, namely the line segment between them.  If $x$ is a boundary point of a cell $\br{w}$ and $y\in\br{w}$ then for any path between $x$ and $y$ that is not contained in $\br{w}$ there is a strictly shorter path which is contained in $\br{w}$.
\end{lemma}
\begin{proof}
The first statement is an obvious consequence of the fact  that the line segment between two boundary points of $\br{w}$ is contained in $\br{w}$ and is a Euclidean geodesic.  The second uses the following observation: if $\gamma$ is a path from $x$ to $y$ that exits $\br{w}$ at a boundary point $z$ then either $\gamma$ re-enters at $z$ in which case it can be shortened by removing the intervening component or it re-enters at another boundary point $\tilde{z}$, in which case it can be shortened by replacing the intervening component by the line segment from $z$ to $\tilde{z}$, as the latter is geodesic.
\end{proof}

This lemma suggests a substantial reduction of the problem. We fix some notation.
\begin{defn}\label{def:bridgepoints}
For distinct points $x$ and $y$, the unique smallest cell that contains both is called the \textit{common cell} of $x$ and $y$.
If $\br{w}$ is a cell, the intersection points of its $n+1$ maximal
subcells, which have addresses $\br{wi\bar j}$ for all pairs $i,j$, are called the \textit{bridge points} of $\br{w}$.
\end{defn}

It is apparent that if $x$ is the boundary point $\br{w\bar{i}}$ of the common cell $\br{w}$ of $x$ and $y$ then $y$ is in a different maximal subcell $\br{wj}$ than $x$, so $j\neq i$.  Any path from $x$ to $y$ must pass through a boundary point of $\br{wj}$, and such have the form $\br{wj\bar{k}}$.  The next lemma shows that we may assume $k=i$.

\begin{lemma} \label{geolemma}
Let  $x=\br{w\bar i}$ and $y\in\br{wj}$  with $j\neq i$. For any path $\gamma$ from $x$ to $y$ there is a path of shorter or equal length that enters $\br{wj}$ through the bridge point $\br{wi\bar j}=\br{wj\bar{i}}$.
\end{lemma}
\begin{proof}
We know $\gamma$ enters $\br{wj}$ at a point $\br{wj\bar{k}}$. Moreover, by modifying $\gamma$ as in Lemma~\ref{lem:geostaysincell} to remove all excursions outside $\br{wj}$, we obtain a (possibly shorter) path with the property that the only portion of the path that is outside $\br{wj}$ is the initial segment from $\br{w\bar i}$ to $\br{wj\bar{k}}$. Let $L$ denote the length of a side of a maximal subcell.   If $k=i$ we are done, but if not then the point $\br{wj\bar k}$ is not in $\br{wi}$, so is distance at least $L$ from $\br{wi}$ and thus $2L$ from $\br{w\bar i}$.  However the line segments from $\br{wj\bar k}$ to $\br{wj\bar i}=\br{wi\bar j}$ to $\br{w\bar i}$ have length exactly $2L$, so using this as the initial segment of $\gamma$ makes the path no longer and ensures it enters $\br{wj}$ at the specified point.
\end{proof}

The preceding lemma tells us how to construct a geodesic from a boundary point of a cell to any point inside the cell.   Moreover, it allows us to write the length of this geodesic in terms of the barycentric coordinates.  The latter was previously noted in~\cites{strichartz1999, christensen2008}.

\begin{prop} \label{dist-bary}
Let $x=\br{w\bar{i}}$ and $y\in\br{w}$. Then $d(x,y)=[x]_i-[y]_i$ and there is a geodesic from $x$ to $y$.
\end{prop}
\begin{proof}
It is sufficient to work on $S_n$, because $x,y\in\br{w}$ implies the first $|w|$ binary terms of $[x]_i$ and $[y]_i$ are equal.  Let  $w_m$ be the level $m$ truncation of an (infinite) address for $y$ and $x_m=\br{w_m\bar i}$, so $x_0=x$.  Define a path $\gamma$ on $[0,1)$ by mapping $[1-2^{-k},1-2^{-(k+1)}]$ to the segment from $x_k$ to $x_{k+1}$.  Since $\{\bar{w_m}\}$ intersects to $y$ we have $x_m\to y$ and thus may extend $\gamma$ continuously to $[0,1]$ by setting $\gamma(1)=y$.  Lemma~\ref{geolemma} ensures any path from $x$ to $y$ is at least as long as $\gamma$, so $\gamma$ is a geodesic.

It remains to compute the length of $\gamma.$  Observe that $\gamma$ is constant on any segment where $x_k=x_{k+1}$, and otherwise has length $2^{-(k+1)}$.  Moreover, $x_k=x_{k+1}$ if and only if the $(k+1)^\text{th}$ letter in the address of $y$ is $i$.  Now Lemma~\ref{lem:dyadbary} says this occurs if and only if the $(k+1)^{\text{th}}$ term in the binary expansion of $[y]_i$ is $1$, and that $x=\br{\bar i}$ implies every term in the binary expansion of $[x]_i$ is $1$.  Thus the binary expansion of $[x]_i-[y]_i$ has coefficient $1$ multiplying $2^{-(k+1)}$ precisely when $x_k\neq x_{k+1}$, proving the formula for $d(x,y)$.
\end{proof}
\begin{remark}
Note that the construction of $\gamma$ terminates (i.e. $x_k=y$ for all sufficiently large $k$) if $y$ is a vertex which has address $y=\br{w'\bar i}$ for some word $w'$.
\end{remark}

\begin{cor}\label{cor:geodesicsexist}
For any distinct pair of points $x$, $y$ there is a geodesic from $x$ to $y$.  All geodesics are contained in the common cell of $x$ and $y$.
\end{cor}
\begin{proof}
Take cells $\br{w_x}$ containing $x$ and $\br{w_y}$ containing $y$ with $|w_x|=|w_y|=m$ large enough that the cells do not intersect.  There is a geodesic from each boundary point of $\br{w_x}$ to $x$ and from each boundary point of $\br{w_y}$ to $y$. Moreover, between a boundary point of $\br{w_x}$ and a boundary point of $\br{w_y}$ any geodesic is composed of a finite union of edges in the graph at level $m$ by Lemma~\ref{geolemma}.  This reduces the problem of finding a geodesic to identifying the shortest curve in a finite collection, which may always be solved.

The fact that geodesics stay in the common cell is a consequence of Lemma~\ref{lem:geostaysincell}, because a path that exits this cell may be written as the concatenation of a path from $x$ to the cell boundary, a path between cell boundary points, and a path from the cell boundary to $y$, each of which can be strictly shortened if it exits the cell.
\end{proof}

\begin{cor} \label{cor:bary-dist-n}
Let $x=\br{\bar i}$, and define a hyperplane $H=\{(y_0,y_1,\ldots ,y_n)\in S_n: y_i=a\}$ for some fixed $a$.
Then $d(x,y)$ is constant for all $y\in H$.
\end{cor}

\subsection{Uniqueness in $S_2$}\label{subsec:uniqueness-S2}

We now turn to the question of geodesic uniqueness in the Sierpinski $n$-gasket.  As shown in
\cite{saltan2016}, there are at most five geodesics between two points in $S_2$. We provide an alternate, 
more geometrical proof, beginning with the case of geodesics between a boundary point 
of a cell and a point contained in the cell.

\begin{prop}\label{prop:two-geodesics-n}
Let $x=\br{\bar i}$ be a boundary point of $S_n$, and $y\neq x$. Then there is a unique geodesic between $x$ and $y$ unless $y=\br{wk\bar j}=\br{wj \bar k}$ for some finite word $w$ and some $j,k$ such that $i$, $j$ and $k$ are distinct.
\end{prop}
\begin{proof}
Let $\gamma_1$ be a geodesic constructed as in Proposition~\ref{dist-bary} from a sequence of cells $\{\br{w_m}\}$ that intersect to $y$ and $\gamma_2$ be any other geodesic from $x$ to $y$. 
Take the largest $m$ such that $\gamma_2$ enters $\br{w_m}$ through a vertex $z'=\br{w_m \bar j}$ with $j\neq i$, and write $z=\br{w_m\bar i}$. Using Lemma~\ref{geolemma} we can modify $\gamma_2$ to form  $\tilde{\gamma}_2$ which passes through $z$ and then $z'$ and has the same length as $\gamma_2$.  Then 
\[ d(x,z)+d(z,y)= |\gamma_1|=|\tilde{\gamma}_2|=d(x,z)+d(z,z')+d(z',y)=d(x,z)+2^{-m}+d(z',y).\]
 However Proposition~\ref{dist-bary} ensures $d(z,y)\leq 2^{-m}=d(z,z')$, so $d(z',y)=0$ and $y=z'$ is a vertex.  Moreover $\tilde{\gamma}_2=\gamma_1$.

Now by our choice of $m$ we know $\gamma_1$ and $\gamma_2$ coincide between $x$ and $\br{w_{m-1}\bar i}$ because they are built from the same bridge points.  Thus the only difference between these geodesics occurs on $\br{w_{m-1}}$, and they begin at $\br{w_{m-1}\bar i}$ and end at $y=\br{w_m \bar j}$, where $w_m=w_{m-1}k$ for some $k$.    Determining the possible paths in $\br{w_{m-1}}$ is therefore the same as determining the geodesics in $S_n$ from $\br{\bar{i}}$ to $\br{k\bar j}=\br{j\bar k}$ for some $k$ and some $j\neq i$. This is an easy finite computation: there is a unique such geodesic if $k=i$ or $k=j$ and exactly two geodesics if $k\neq i,j$, one through $\br{i\bar k}=\br{k\bar i}$ and one through $\br{i\bar j}=\br{j\bar i}$.
\end{proof}
\begin{remark}
If $y$ is a vertex of the form identified in the proposition then  we may use either of its two addresses to construct a geodesic from $x$ to $y$ by the method of Proposition~\ref{dist-bary}.  The two addresses lead to the two distinct geodesics identified in Proposition~\ref{prop:two-geodesics-n}.
\end{remark}


\begin{cor} \label{cor:address-nonunique}
Let $y \in \br{w} \subseteq S_2$. If there are two distinct geodesics from $y$ to a boundary point of $\br{w}$  
then there is only one geodesic from $y$ to each of the other two boundary points of $\br{w}$.
\end{cor}
\begin{proof}
There are two distinct geodesics from $y$ to $\br{w\bar{i}}$, so by Proposition~\ref{prop:two-geodesics-n}
 we have $y=\br{wj\bar k}$ where $j$ and $k$ are distinct from each other and from $i$.  The proposition also tells us that such a $y$ has unique geodesics to the boundary points $\br{w\bar j}$ and $\br{w\bar k}$, and since we are in $S_2$ with $i,j,k$ distinct, this covers all boundary points.
\end{proof}

Our results on geodesics between a point in a cell and a boundary point of that cell have implications for geodesics between arbitrary points.  Recall that the bridge points of $S_n$ are the intersection points of maximal subcells.  The following lemma gives a useful classification of geodesics between points in distinct maximal cells according to the number of bridge points they contain.
 
\begin{lemma}\label{lem:1or2bridge}
Any geodesic between two points in $S_n$ passes through at most two bridge points.
\end{lemma}
\begin{proof}
Let $x\in \br{i}$ and $y\in \br{j}$.  We may construct a path between them by concatenating 
a geodesic from $x$ to $\br{i \bar j}$ and a geodesic from $\br{i \bar j}$ to $y$. By Proposition~\ref{dist-bary} each such geodesic has length at most $\half$, so $d(x,y)\leq 1$.  However, bridge points are separated by distance $\half$, so there can be at most two on a geodesic.
\end{proof}

We note that if the geodesic from $x$ to $y$ passes through only one bridge point it must be the intersection point of  the maximal cells containing them.

\begin{defn} \label{P1-P2}
Let $x \in \br{i}$ and $y \in \br{j}$, $i\neq j$. The geodesic $\gamma$ from $x$ to $y$ is a  \textit{$P_1$ geodesic} if 
it passes through the bridge point $\br{i\bar j}=\br{i}\cap\br{j}$ and a  \textit{$P_2$ geodesic} if it passes through two bridge 
points, $\br{i\bar k}$ and $\br{j\bar k}$, where $k \neq i,j$. 
\end{defn}

Cristea and Steinsky~\cite{cristea2013} provide geometric criteria for $S_2$ and $S_3$ that 
allow one to determine whether one or both types of geodesics exist between some pair of points. The proof of the following theorem, which gives a sharp bound on the number of geodesics between $x$ and $y$ in $S_2$, recovers their results for $S_2$.  The sharp bound was previously proved in~\cite{saltan2016} by a different method.

\begin{thm}\label{thm:five}
There are at most five distinct geodesics between any two points in $S_2$, and this bound is sharp.
\end{thm}
\begin{proof}
Fix $x$ and $y$.  Corollary~\ref{cor:geodesicsexist} tells us that all geodesics between $x$ and $y$ lie in their common cell, so we may assume the common cell is $S_2$.  Thus $x\in\br{i}$ and $y\in\br{j}$ with  $i\neq j$, and the geodesics between them are either $P_1$ geodesics through $\br{i\bar j}$ or $P_2$ geodesics through $\br{i\bar k}$ and $\br{j\bar l}$ where $k\neq i,j$.

If $\gamma$ is a geodesic from $x$ to $y$ then its restriction to $\br{i}$ is a geodesic from $x$ to either $\br{i\bar j}$ or $\br{i\bar k}$.  Proposition~\ref{prop:two-geodesics-n} says there are at most two geodesics to either of these points and Corollary~\ref{cor:address-nonunique} says that if there are two to one such point then there is only one to the other, so there are at most three options for the restriction of $\gamma$ to $\br{i}$. Similarly there are at most three options for the restriction of $\gamma$ to $\br{j}$.

We now consider how the pieces of geodesic previously described may be combined. 
\begin{enumerate}
\item[Case 1]  There are two distinct geodesics between $x$ and $\br{i\bar{j}}$ and two between $y$ and $\br{i\bar{j}}$, providing  four $P_1$ paths. In this case the geodesics from $x$ to $\br{i\bar k}$ and $y$ to $\br{j\bar k}$ are unique, as is that between these bridge points, so there is one $P_2$ path.  If the $P_1$ and $P_2$ geodesics are the same length then there are five geodesics in total, otherwise there are four or one.    Figure \ref{fig:5-geo-fig} shows that five may be achieved.
\item[Case 2] There are two distinct geodesics between $x$ and $\br{i\bar{j}}$ and two between $y$ and $\br{j\bar{k}}$. Then there is one from $y$ to $\br{i\bar j}$ so there are two $P_1$ paths, and there is one from $x$ to $\br{i\bar k}$, so there are two $P_2$ paths. If all of these have the same length there are four geodesics, otherwise there are two.
\item[Case 3] There is only one geodesic from $x$ to each of $\br{i\bar{j}}$ and $\br{i\bar{k}}$ and only one from $y$ to each of $\br{j\bar{i}}$ and $\br{j\bar{k}}$. Then there is one $P_1$ and one $P_2$ path; if they are the same length there are two geodesics, and otherwise there is one. \qedhere
\end{enumerate}
\end{proof}

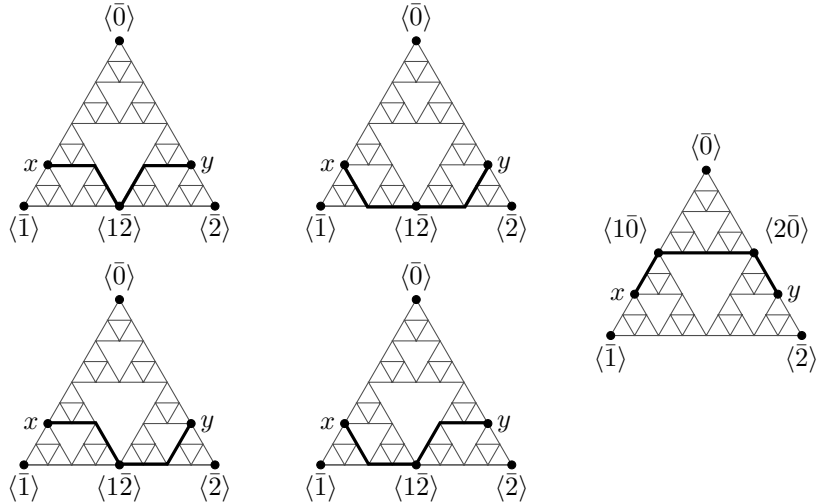
\begin{figure}[H]
\centering
\begin{minipage}{0.55\textwidth}
\begin{subfigure}[b]{0.3\textwidth}
  \centering
  \begin{tikzpicture}[x=1in,y=0.5*\sqTHREE in]
  \Sierpinski[1in]{3}
  \filldraw (1/2,1/8) node[above]{};
  \draw[very thick] (0,-0.005) +(1/8-0.005,1/4) -- +(3/8-0.005,1/4) -- +(1/2-0.005,0) 
    -- +(1/2+0.005, 0) -- +(5/8+0.005,1/4) -- +(7/8+0.005,1/4);
  \filldraw (1/8,1/4) circle (1.5pt) node[left] {$x$};
  \filldraw (7/8,1/4) circle (1.5pt) node[right] {$y$};
  \filldraw (1/2,1) circle (1.5pt) node[align=center,above] {$\br{\bar 0}$};
  \filldraw (0,0) circle (1.5pt) node[below] {$\br{\bar 1}$};
  \filldraw (1,0) circle (1.5pt) node[below] {$\br{\bar 2}$};
  \filldraw (1/2,0) circle (1.5pt) node[align=center,below] {$\br{1 \bar 2}$};
  \end{tikzpicture}
\end{subfigure}
\hspace{4em}
\begin{subfigure}[b]{0.3\textwidth}
  \centering
  \begin{tikzpicture}[x=1in,y=0.5*\sqTHREE in]
  \Sierpinski[1in]{3}
  \filldraw (1/2,1/8) node[above] {};
  \draw[very thick] (0,-0.005) +(1/8-0.005,1/4) -- +(1/4-0.005,0)     -- +(1/2,0) -- +(3/4+0.005,0) -- +(7/8+0.005,1/4);
  \filldraw (1/8,1/4) circle (1.5pt) node[left] {$x$};
  \filldraw (7/8,1/4) circle (1.5pt) node[right] {$y$};
 \filldraw (1/2,1) circle (1.5pt) node[align=center,above] {$\br{\bar 0}$};
  \filldraw (0,0) circle (1.5pt) node[below] {$\br{\bar 1}$};
  \filldraw (1,0) circle (1.5pt) node[below] {$\br{\bar 2}$};
  \filldraw (1/2,0) circle (1.5pt) node[align=center,below] {$\br{1 \bar 2}$};
  \end{tikzpicture}
\end{subfigure}
\\
\begin{subfigure}[b]{0.3\textwidth}
  \centering
  \begin{tikzpicture}[x=1in,y=0.5*\sqTHREE in]
  \Sierpinski[1in]{3}
  \filldraw (1/2,1/8) node[above]{};
  \draw[very thick] (0,0.005) +(1/8,1/4) -- +(3/8,1/4) -- +(1/2,0) -- +(3/4,0) -- +(7/8,1/4);
  \filldraw (1/8,1/4) circle (1.5pt) node[left] {$x$};
  \filldraw (7/8,1/4) circle (1.5pt) node[right] {$y$};
 \filldraw (1/2,1) circle (1.5pt) node[align=center,above] {$\br{\bar 0}$};
  \filldraw (0,0) circle (1.5pt) node[below] {$\br{\bar 1}$};
  \filldraw (1,0) circle (1.5pt) node[below] {$\br{\bar 2}$};
  \filldraw (1/2,0) circle (1.5pt) node[align=center,below] {$\br{1 \bar 2}$};
  \end{tikzpicture}
\end{subfigure}
\hspace{4em}
\begin{subfigure}[b]{0.3\textwidth}
  \centering
  \begin{tikzpicture}[x=1in,y=0.5*\sqTHREE in]
  \Sierpinski[1in]{3}
  \filldraw (1/2,1/8) node[above]{};
  \draw[very thick] (0,0.005) +(1/8,1/4) -- +(1/4,0) -- +(1/2,0) -- +(5/8,1/4) -- +(7/8,1/4);
  \filldraw (1/8,1/4) circle (1.5pt) node[left] {$x$};
  \filldraw (7/8,1/4) circle (1.5pt) node[right] {$y$};
 \filldraw (1/2,1) circle (1.5pt) node[align=center,above] {$\br{\bar 0}$};
  \filldraw (0,0) circle (1.5pt) node[below] {$\br{\bar 1}$};
  \filldraw (1,0) circle (1.5pt) node[below] {$\br{\bar 2}$};
  \filldraw (1/2,0) circle (1.5pt) node[align=center,below] {$\br{1 \bar 2}$};
  \end{tikzpicture}
\end{subfigure}
\end{minipage}
\begin{minipage}{0.25\textwidth}
\begin{subfigure}[b]{0.3\textwidth}
  \centering
  \begin{tikzpicture}[x=1in,y=0.5*\sqTHREE in]
  \Sierpinski[1in]{3}
  \filldraw (1/2,1/2) node[below]{};
  \draw[very thick] (1/8,1/4) -- (1/4,1/2) -- (3/4,1/2) -- (7/8,1/4);
  \filldraw (1/8,1/4) circle (1.5pt) node[left] {$x$};
  \filldraw (7/8,1/4) circle (1.5pt) node[right] {$y$};
  \filldraw (1/2,1) circle (1.5pt) node[align=center,above] {$\br{\bar 0}$};
  \filldraw (0,0) circle (1.5pt) node[below] {$\br{\bar 1}$};
  \filldraw (1,0) circle (1.5pt) node[below] {$\br{\bar 2}$};
  \filldraw (1/4,1/2) circle (1.5pt) node[above left] {$\br{1\bar 0}$};
  \filldraw (3/4,1/2) circle (1.5pt) node[above right] {$\br{2\bar 0}$};
  \end{tikzpicture}
\end{subfigure}
\end{minipage}
\caption{Points $x$ and $y$ that are connected by five distinct geodesics: four $P_1$ geodesics (left)
  and one $P_2$ geodesic (right).}
  \label{fig:5-geo-fig}
\end{figure}

\subsection{Uniqueness in $S_n$} \label{subsec:uniqueness-Sn}

The proof of Theorem~\ref{thm:five} relies on two facts about $S_2$:  there are only three bridge points, so there is at most one pair of bridge points through which a $P_2$ geodesic can pass, and nonuniqueness of a geodesic to one bridge point implies uniqueness to the other two bridge points (Corollary~\ref{cor:address-nonunique}). Neither of these arguments is directly applicable to $S_n$, $n\geq 3$.  However we can obtain a sharp bound in this more general setting by making a more detailed analysis of $P_2$ geodesics.

\begin{lemma} \label{2P2}
Let $x\in\br{i}$ and $y\in\br{j}$, where $i\neq j$.  Then there exist $P_2$ geodesics between
$x$ and $y$ passing through at most two distinct pairs of bridge points.
\end{lemma}
\begin{proof}
Suppose there is a $P_2$ geodesic $\gamma$ from $x$ to $y$ through $\br{i\bar k}$ and $\br{j\bar k}$.  Applying Proposition \ref{dist-bary} we have $d(x, \br{i\bar k})=[\br{i\bar k}]_k-[x]_k=\half-[x]_k$ and similarly for $d(y,\br{j \bar k})$, so
\begin{equation*}
 d(x,y)=|\gamma|
= \half+ d(x,\br{i\bar k})+d(y,\br{j \bar k})
 = \frac{3}{2}-[x]_{k}-[y]_{k}.
\end{equation*}
However in the proof of Lemma~\ref{lem:1or2bridge} we saw that $d(x,y)\leq 1$, so $[x]_k+[y]_k\geq\half$.

Now $\sum_{k=0}^n [x]_k+ [y]_k =2$ from the definition of the barycentric coordinates, and $x\in\br{i}$, $y\in\br{j}$ implies $[x]_i\geq\half$ and $[y]_j\geq\half$, so the number of $k$ for which $[x]_k+[y]_k\geq\half$ is at most $2$, which implies there are at most two values of $k$ for which there is a $P_2$ geodesic through the bridge points $\br{i\bar k}$ and $\br{j\bar k}$.
\end{proof}

It is apparent in the preceding proof that the locations of $x$ and $y$ are tightly constrained when they admit geodesics through two distinct pairs of bridge points. The following result makes this precise.

\begin{lemma} \label{lemma:P2-unique}
Let $x\in \br{i}$ and $y\in \br{j}$ with $i \neq j$. If there are $P_2$ geodesics from 
$x$ to $y$ through two distinct pairs of bridge points, then there are exactly two $P_2$ geodesics from $x$ to $y$.
\end{lemma}
\begin{proof}
Let $\gamma_1$ be a $P_2$ geodesic passing through $\br{i\bar k}$ and $\br{j\bar k}$, and let 
$\gamma_2$ be a $P_2$ geodesic passing through $\br{i\bar l}$ and $\br{j\bar l}$. It follows from the proof of Lemma~\ref{lem:1or2bridge} that $|\gamma_1|=|\gamma_2|\leq 1$, so $|\gamma_1|+
|\gamma_2|\leq 2$.  Now by the triangle inequality
\begin{align*}
2&= 1+d(\br{i\bar{k}},\br{i\bar{l}})+d(\br{j\bar{k}},\br{j\bar{l}})\\
&\leq \half +d(\br{i\bar{k}},x)+d(x,\br{i\bar{l}}) + \half + d(\br{j\bar{k}},y)+d(y,\br{j\bar{l}})\\
&=  |\gamma_1|+|\gamma_2|
\leq2.
\end{align*}
Thus equality holds in the triangle inequality and $x$ lies on the geodesic connecting $\br{i\bar{l}}$ to $\br{i\bar{k}}$, and $y$ lies on the geodesic connecting $\br{j\bar{l}}$ to $\br{j\bar{k}}$, both of which are lines in $\mathbb{R}^n$. This shows the geodesics from $x$ to $\br{i\bar k}$ and to $\br{i\bar l}$ are unique and similarly for $y$; the result then follows from Lemma~\ref{2P2}
\end{proof}

Lemmas~\ref{2P2} and~\ref{lemma:P2-unique} provide us with sufficient restrictions on the number
of geodesics to prove our main result on geodesics in $S_n$, $n\geq3$.

\begin{thm} \label{thm:eight}
There exist at most eight distinct geodesics between any two points in $S_n$, $n\geq3$ and this bound is sharp.
\end{thm}
\begin{proof}
Let $x\in\br{i}$ and $y\in\br{j}$. As in the proof of Theorem~\ref{thm:five} we may assume $i\neq j$.
By Proposition~\ref{prop:two-geodesics-n}, there exist at most two geodesics from $x$ to $\br{i\bar j}$ and 
two geodesics from $\br{i\bar j}$ to $y$.  Concatenations of these pairs of geodesics yield a maximum of 
four $P_1$ geodesics.

By Lemma \ref{2P2}, there exist $P_2$ geodesics through at most two distinct pairs of bridge points.  If there are $P_2$ geodesics through two distinct pairs of bridge points, then by Lemma~\ref{lemma:P2-unique}, there are exactly two $P_2$ geodesics between $x$ and $y$. If, in addition, there exist $P_1$ geodesics, then there are at most six total geodesics.

If, instead, there are $P_2$ geodesics only through a single pair of bridge points $\br{i\bar k}$, $\br{j\bar k}$, then they are obtained by concatenating geodesics from $x$ to $\br{i\bar k}$ (of which there are at most two by Proposition~\ref{prop:two-geodesics-n}) and $y$ to $\br{j\bar k}$ (of which there are again at most two) with the interval from  $\br{i\bar k}$ to $\br{j\bar k}$.  This yields at most four $P_2$ geodesics between $x$ and $y$.  If, in addition, there exist $P_1$ geodesics, then there are at most eight geodesics in total.  Sharpness is demonstrated by Example~\ref{eg:eight}, which is written in $S_3$ and embeds in $S_n$ for all $n\geq 3$.
\end{proof}

\begin{example}\label{eg:eight}
Let $x = \br{202\bar 1} = \br{201 \bar 2}$ and $y = \br{303 \bar 1} = \br{301 \bar 3}$. 
Then the following are the geodesics from $x$ to $y$:
\begin{align*}
\gamma_1 &: x\to\br{202\bar3}\to\br{20\bar3}\to\br{2\bar3}=\br{3\bar2}\to\br{30\bar2}\to\br{303\bar2}\to y, \\
\gamma_2 &: x\to\br{202\bar3}\to\br{20\bar3}\to\br{2\bar3}=\br{3\bar2}\to\br{30\bar2}\to\br{301\bar2}\to y, \\
\gamma_3 &: x\to\br{201\bar3}\to\br{20\bar3}\to\br{2\bar3}=\br{3\bar2}\to\br{30\bar2}\to\br{303\bar2}\to y, \\
\gamma_4 &: x\to\br{201\bar3}\to\br{20\bar3}\to\br{2\bar3}=\br{3\bar2}\to\br{30\bar2}\to\br{301\bar2}\to y, \\
\gamma_5 &: x\to\br{202\bar0}\to\br{20\bar0}=\br{2\bar0}\to\br{3\bar0}=\br{30\bar0}\to\br{303\bar0}\to y, \\
\gamma_6 &: x\to\br{202\bar0}\to\br{20\bar0}=\br{2\bar0}\to\br{3\bar0}=\br{30\bar0}\to\br{301\bar0}\to y, \\
\gamma_7 &: x\to\br{201\bar0}\to\br{20\bar0}=\br{2\bar0}\to\br{3\bar0}=\br{30\bar0}\to\br{303\bar0}\to y, \\
\gamma_8 &: x\to\br{201\bar0}\to\br{20\bar0}=\br{2\bar0}\to\br{3\bar0}=\br{30\bar0}\to\br{301\bar0}\to y,
\end{align*}
where each geodesic is composed of the edges in $S_n$ joining each pair of consecutive points.

Note that each portion of these connecting $x$ or $y$ to a bridge point is geodesic because it is constructed by the algorithm in Proposition~\ref{dist-bary}.  The first four are $P_1$ paths, with length $2(\frac18+\frac18+\frac14)=1$.  The second four are $P_2$ paths, with length $2(\frac18+\frac18)+\frac12=1$.  We have constructed all the candidates to be geodesic from $x$ to $y$ (as described in Corollary~\ref{cor:geodesicsexist}), so the fact that they are all the same length ensures all are geodesics.
\end{example}

For use in later sections of this paper, it is convenient to note the following consequence of our analysis of the structure of geodesics.

\begin{thm}\label{thm:nullset}
The set of pairs of points $(x,y)\in S_n\times S_n$ such that there is more than one geodesic from $x$ to $y$ has zero $\mu_n\times\mu_n$-measure.
\end{thm}
\begin{proof}
First observe that sets of the form $\{a\}\times S_n$ or $S_n\times \{b\}$ are null for $\mu_n\times\mu_n$.  Taking the countable union over $a\in V_n^*$ and over $b\in V_n^*$ gives a null set.  Now observe from Proposition~\ref{prop:two-geodesics-n} that $x$ is connected to any boundary of a cell containing $x$ by more than one geodesic then $(x,y)$ is in one of these null sets and similarly for $y$.  Accordingly, we can assume that there is a unique geodesic from $x$ to any boundary point of a cell containing $x$, and similarly for $y$.

In this circumstance the only way $x$ and $y$ can be joined by more than one geodesic involves at least one $P_2$ geodesic.  Precisely,  there is a cell $\br{w_x}$ containing $x$ and a cell $\br{w_y}$ containing $y$, these cells are joined by distinct geodesics $\gamma$ and $\gamma'$ such that $\gamma$ enters $\br{w_x}$ at $a_x$ and $\br{w_y}$ at $a_y$, and $\gamma'$ enters $\br{w_x}$ at $a'_x$ and $\br{w_y}$ at $a'_y$.  Moreover the fact that these have equal length may be written as $d(x,a_x)+d(y,a_y)+|\gamma|=d(x,a'_x)+d(y,a'y)+|\gamma'|$. There are countably many choices of pairs $\br{w_x}$, $\br{w_y}$ and, for each pair, finitely many possibilities for $\gamma$ and $\gamma'$, so to prove the set of pairs $(x,y)$ joined by non-unique geodesics of this type is null we need only prove that for such a pair of cells, geodesics and boundary points, one has
\begin{equation}\label{eqn:nullseteqn}
	 (\mu_n\times\mu_n)\bigl( \{(x,y): d(x,a_x)-d(x,a'_x)+d(y,a_y)-d(y,a'_y)=|\gamma'|-|\gamma|\}\bigr)=0.
	\end{equation}
Moreover, since $\mu_n\times\mu_n$ is a product measure, by Fubini's theorem it is sufficient that this set has zero $\mu_n$ measure for each fixed $y$. More precisely, since fixing $y$ fixes the value $d(y,a_y)-d(y,a'_y)$, it is enough that for any $s$,  $\mu_n\bigl(\{ x: d(x,a_x)-d(x,a'_x)=s\}\bigr)=0$.  Clearly the question of whether this set is null is invariant under rescaling the cell $\br{w_x}$ to be $S_n$, and by symmetry we may assume $a_x=q_0$, $a'_x=q_1$; we assume this is the case.

Now we use Proposition~\ref{dist-bary} to write $d(x,q_0)$ as the projection on the barycentric coordinate corresponding to $q_0$, and writing $d(x,q_1)$ in the same manner we find that $d(x,q_0)-d(x,q_1)$ is the projection of $x-q_0$ on the unit vector $q_1-q_0$ which is parallel to an edge of the simplex.  Parametrizing the position along the line from $q_0$ to $q_1$ by $[0,1]$ and writing $\pi:S_n\to[0,1]$ for the projection on $q_1-q_0$ in this parametrization, we see that the measure $\mu_n\bigl(\{ x: d(x,q_0)-d(x,q_1)\in E\}\bigr)=\mu_n\circ\pi^{-1}(E)$ is the pushforward measure $\pi_*\mu_n$ of $\mu_n$ under $\pi$.  However this is a self-similar measure on $[0,1]$ by Proposition~\ref{prop:general-pushforward}, with the self-similarity relation
\begin{equation*}
	\pi_*\mu_n(s) = \frac1{n+1} \Bigl( \pi_*\mu_n (2s) + \pi_*\mu_n(2s-1) + (n-1) \pi_*\mu_n\bigl(2s-\frac12\bigr) \Bigr).
	\end{equation*}
This measure is non-atomic.  To see this, suppose the contrary.  It is a probability measure, so there is an atom which attains the maximal mass among atoms; we let $s_0$ be the location of such an atom. Then the self-similarity relation says
\begin{equation*}
	(n+1)\pi_*\mu_n(\{s_0\}) = \pi_*\mu_n (\{2s_0\}) + \pi_*\mu_n(\{2s_0-1\}) + (n-1) \pi_*\mu_n\bigl(\{2s_0-\frac12\}\bigr)
	\end{equation*}
but at most two of the points $2s_0$, $(2s_0-1)$ and $(2s_0-\frac12)$ are in $[0,1]$, and the atoms at these points have mass not exceeding $\mu_n(\{s_0\})$, so that
\begin{equation*}
	(n+1)\pi_*\mu_n(\{s_0\}) \leq n \pi_*\mu_n(\{s_0\}) 
	\end{equation*}
and thus $\pi_*\mu_n(\{s_0\})=0$.

The fact that $\pi_*\mu_n$ is non-atomic says precisely that  $\mu_n\bigl(\{ x: d(x,q_0)-d(x,q_1)=s\}\bigr)=0$ for every choice of $s$. As previously noted, this ensures the measure of the set in~\eqref{eqn:nullseteqn} is zero, and by taking the union over the countably many possible cells and geodesics connecting their boundary points we complete the proof.
\end{proof}

\begin{remark}
The paper~\cite{GuYeXi} states Theorem~\ref{thm:nullset} as their Theorem~1.3, but it appears to us that something is missing in the proof.  Specifically, the authors reduce to the situation where, in our notation, $x\in\br{0}$, $y\in\br{1}$ and there are a $P_1$ and a $P_2$ geodesic between these points (see the reasoning following their Lemma~4.3, where they say that there are vertices from our $V_1$ which they call $b_1$, $b_i$ and $b'_i$, such that 
$$ d(x,b_1)+d(b_1,y)=d(x,b_i)+d(b_i,b'_i)+d(b_i,y)$$
with $d(b_i,b'_i)=\frac12$.)  Using their Lemmas~4.4--4.7 they appear to be saying, in the proof of Proposition~4.8, that then $x$ and $y$ are points of $V_*$.  (They write this as $x=\sigma1^\infty$, $y=\sigma'0^\infty$.)  Yet we can give an example of points $x$ and $y$ in $S_2$ that are as described above but are not from $V_*$, as follows.  Consider the three line segments forming a triangle around the central hole of the gasket $S_2$. These have vertices with addresses $\br{0\bar{1}}$, $\br{0\bar{2}}$ and $\br{1\bar{2}}$.  Take a point $x$ in $\br{0}$ at distance $s$ from $\br{0\bar{2}}$ where $s$ is not a dyadic rational (so $x\notin V_*$) and $s<\frac14$.  Take $y$ in $\br{1}$ at distance $\frac14-s$ from $\br{1\bar{2}}$. Evidently, any geodesic between these points lies on the three line segments.  Now the distance from $x$ to $y$ through the points $\br{0\bar{2}}=\br{2\bar{0}}$ and $\br{2\bar{1}}=\br{1\bar{2}}$ is $s+\frac12+(\frac14-s)=\frac34$, because it includes the edge through the cell $\bar{2}$.  It is equally apparent that the distance via $\br{0\bar{1}}=\br{1\bar{0}}$ is $\frac12-s + \frac12-(\frac14-s)=\frac34$.  So there are two geodesics joining these points but neither point is in $V_*$.
\end{remark}


\section{Interpolation}\label{sec:interp}

In Euclidean space the barycenter of sets $A$ and $B$ is
$(1-t)A+tB=\{(1-t)a+tb:a\in A, b\in B\}$.  In a geodesic space the natural
analogue, introduced in~\cite{cordero2001}, is the set defined by
\begin{equation}\label{eq:Zt}
	\begin{split}
	\tilde Z_t(a,b)&=\{x: d(a,x)=td(a,b) \text{ and } d(x,b)=(1-t)d(a,b)\},\\
	\tilde Z_t(A,B)&= \{ \tilde Z_t(a,b):a\in A,b\in B\}.
	\end{split}
\end{equation}
From our results on geodesics we know that $\tilde Z_t(a,b)$ is a single point for
almost all $a$ and $b$, but in any case contains at most $8$ points.

The classical Brunn-Minkowski inequality in $\mathbb{R}^n$ says
$\vol(\tilde{Z}_t(A,B))^{1/n}\geq (1- t)\vol(A)^{1/n}+t\vol(B)^{1/n}$.  This convexity
result has many applications, for which we refer to the survey~\cite{Gardner}.
Our first result makes it clear that no such result can be true for the
self-similar measure $\mu_n$ on $S_n$.

\begin{prop}\label{prop:onedim}
    For $A,B\subset S_n$ the set $\cup_{t\in(0,1)} \tilde Z_t(A,B)$ has
    Hausdorff dimension at most $1$ and hence $\mu_n$-measure zero.
\end{prop}
\begin{proof}
    We have shown that all geodesics are constructed as in
    Corollary~\ref{cor:geodesicsexist} using the method from the proof of
    Proposition~\ref{dist-bary}.  In that argument, the set
    $\cup_{t\in(0,1)}\tilde{Z}_t(a,b)$ lies entirely on the countable collection of
    Euclidean line segments joining vertices from $V_n^*$, no matter what $a$
    and $b$ are.  Hence $\cup_{t\in(0,1)} \tilde Z_t(A,B)$ also lies in this countable
    collection of Euclidean line segments, which is a set of Hausdorff dimension
    $1$.
\end{proof}

We note in passing an amusing consequence of the preceding proof which
emphasizes the difference with the Euclidean case. We call a set $A$ convex when
$a,b\in A$ implies $\tilde Z_t(a,b)\subset A$ for all $t$, and observe that the
intersection of convex sets is convex, so $A$ has a smallest closed convex
superset, called its convex hull.
\begin{cor}
The convex hull of a closed set $A$ has the same $\mu_n$ measure as $A$.
\end{cor}
\begin{proof}
The essential idea of the proof is to take a closed convex set by adjoining to
$A$ a portion of the union of line segments in Proposition~\ref{prop:onedim}
which accumulates only at $A$.

Given $A$, let $V=V_n^*\cap(\cup_{t\in(0,1)} \tilde Z_t(A,A)$ be the set of
vertices on geodesics between points of $A$.  Observe that if $x\in V$ and
$d(x,A)\geq\delta>0$ then from the construction of all geodesics in
Corollary~\ref{cor:geodesicsexist} and Proposition~\ref{dist-bary} it must
be that $x$ lies on the edge of a cell which intersects $A$ and has size at
least $\delta$.   Let $B$ consist of all geodesics between all pairs of
points in $V$.  Note that if $x\in B$ and $d(x,A)\geq\delta$ it lies on one
of the finitely many edges of cells of size at least $\delta$, and thus
$\{b\in B:d(b,A)\geq\delta\}$ is closed. It  follows that any accumulation
point of $B$ that is at a positive distance from $A$ is in $B$, and thus
that $A\cup B$ is closed.

Let us consider the geodesics between points of $A\cup B$.  If $a,b\in B$ then
the geodesic from $a$ to $b$ is in $B$ because it is a subset of a geodesic
between points of $V$.  If $a,b\in A$ then the geodesic between them is the
increasing union of geodesics between points of $V$ with ends that
accumulate to $a$ and $b$, so is also in $A\cup B$.  Similarly, if $a\in A$
and $b\in B$ the geodesic between them is a union of this type except that
one part of the geodesic between points of $V$ is terminated at $b$.  So
$A\cup B$ is convex.

Since $A\cup B$ is closed and convex it must contain the convex hull of $A$.
However, $B$ is a countable union of line segments, so it is one-dimensional
and $\mu_n(A\cup B)=\mu_n(A)$. 
\end{proof}

Proposition~\ref{prop:onedim} tells us that there is no hope that a power of
$\mu_n(\tilde Z_t(A,B))$ is convex in $t$, but it remains possible that the
measure-theoretic properties of $\tilde Z_t(A,B)$ reflect some aspects of the
geometric structure of the Sierpinski gasket $S_n$. We record some definitions
and basic notions that are useful in investigating this question.

\begin{defn} \label{def:commonpath}
    Let $A,B\subseteq S_n$. A  \textit{common path} $\hat\gamma:[0,1]\to S_n$
    from $A$ to $B$ is a finite length path  such that  for each $a\in A$ and
    $b\in B$ there is a  geodesic $\hat\gamma_{a\to b}$ and $t_1,t_2\in[0,1]$,
    called the entry and exit times, with $\hat\gamma=\hat\gamma_{a\to b}([t_1,
    t_2])$. The \textit{initial} and \textit{final entry times} of
    $\hat{\gamma}$ are, respectively, the infimium $t^i_1$ and supremum  $t^f_1$
    of the set of entry times over $a\in A$ and $b\in B$. The \textit{initial}
    and \textit{final exit times} $t^i_2$ and $t^f_2$ are similarly defined from
    the set of exit times.
    If $\hat{\gamma}$ is a maximal common path under inclusion and $t_1^f<t_2^i$,
    we call $\hat{\gamma}$ a \textit{regular common path}.
\end{defn}

For arbitrary $A$ and $B$ there need not be a regular common path, but in many
simple cases either there is such a path or there is a natural way to decompose
$A$ and $B$ so as to obtain such paths between components. Indeed, it is easy to
see that there is a regular common path between disjoint cells that are
sufficiently small compared to their separation.  Using this, if $A$ and $B$ are
disjoint one may take a union of cells covering $A$ and another union of cells
covering $B$ so that any pair of cells, one from the first union and the other
from the second, admits a regular common path.  Of equal importance is the fact
that understanding regular common paths is sufficient for studying some aspects
of the transport of measure via the set $\tilde Z_t(A,B)$, at least for fairly
simple choices of $A$ and $B$.  One way to see this is as follows.  Begin by
deleting a nullset of $X\times X$ from Theorem~\ref{thm:nullset} so that
geodesics are unique, then observe that if $A$ and $B$ are separated by a
distance $3\epsilon>0$ then these geodesics must each contain one of finitely
many edges of size bounded below by $\epsilon$. It follows that $A\times B$ can
be decomposed into finitely many sets $A_j\times B_j$ so that geodesics from
$A_j$ to $B_j$ have some piece of common path.  See also the remark following
Definition~\ref{eta-def}.

When there is a common path it is natural to consider only that part of $\tilde{Z}_t$ that 
lies on the common path.
\begin{defn}
For $A$ and $B$ that admit a common path and $a\in A$, $b\in B$, define a 
modified interpolant $Z_t(a,b)=\hat\gamma_{a\to b}(t)$, i.e., the point in $\tilde{Z}_t$
which lies on the geodesic included in the common path.
\end{defn}

This modified interpolant is  is indeed a function, and for all $t\in[t_1^f,t_2^i]$ we have
$Z_t(A,B)\subset\hat{\gamma}$, as described in the following result.

\begin{prop} \label{interval-prop}
Let $A, B$ be connected subsets of $S_n$ for which  $\hat\gamma$ is a regular
common path. For each $t\in[t^f_1,t^i_2]$ there exists an interval $I_t
\subseteq [0,1]$ such that $Z_t(A,B) = \hat\gamma(I_t)$.
\end{prop}
\begin{proof}
Fix $t\in[t^f_1,t^i_2]$. Continuity of $d(x,y)$ implies $Z_t(a,b)$ is continuous
on the connected set $A\times B$ and thus $Z_t(A,B)$ is a connected subset
of $\hat{\gamma}$. Such subsets have the stated form.
\end{proof}

\begin{defn}\label{def:Ht}
In the circumstances of Proposition~\ref{interval-prop}, let $H_t:(0,1)\to
Z_t(A,B)=\hat\gamma(I_t)$ be the parametrization obtained from the
increasing linear surjection $(0,1)\to I_t$, which may also be defined at
$0$ or $1$, followed by $\hat\gamma$.
\end{defn}

Motivated by the Brunn-Minkowski inequality, our basic object of study will be
the measure on $\hat{\gamma}$ that is induced by the natural measures on $A$ and
$B$ via the interpolant $Z_t$.  In the next two sections we consider two basic
cases: when $A$ is a  cell with measure $\mu_n$ and $B=\{b\}$ is a point with
Dirac mass, and when $A$ and $B$ are both cells with measure $\mu_n$.

\section{Cell-to-point Interpolation of Measure}
\label{sec:interp-point}

In this section we consider $Z_t(A,b)$, where $A$ is a cell not containing $b$,
for which we use the notation $Z_{t,b}(a)=Z_t(a,b)$.  We assume that all
geodesics from points $a\in A=\br{w}$ to $b$ pass through a single boundary
point $\dot a$, which we call the {\em  entry point} of $A$ and that the
geodesic $\hat{\gamma}$ from $\dot a$ to $b$ is unique.   Then  $\hat{\gamma}$
is a common path from $A$ to $b$, and we note that in this situation all exit
times coincide,  $t_2^i=t_2^f=1$, so $\hat{\gamma}$ is regular.  Then for
$t\in[t_1^f,1]$ we know from Proposition~\ref{interval-prop} that $Z_t(A,b)$ is
an interval and $Z_{t,b}$ is a function, permitting us to study interpolation by
considering the pushforward of $\mu_n$ under $Z_{t,b}$, as in the following
definition. Note, too,  that $t_1^f=(1+2^{|w|}d(\dot a,b))^{-1}$.  

\begin{defn} \label{eta-def}
For $t\in[0,1]$ let $\eta_t(X) = \mu_n(Z_{t,b}^{-1}(X))$,
for all Borel sets $X\subseteq\hat\gamma$.
\end{defn}

\begin{remark}
    The definition of $\eta_t$ depends on our assumptions regarding the common
    path, and one might think this could be avoided by instead studying
    something like $\eta'_t(X) = \mu_n\bigl(\{a:Z_t(a,b)\cap
    X\neq\emptyset\}\bigr)$.  However, a little thought shows that there is not
    much loss of generality in studying the simpler quantity $\eta_t$ instead.
    We made two assumptions: that all geodesics from points $a\in A=\br{w}$ to
    $b$ pass through a single boundary point $\dot a$ of $A$,  and that the
    geodesic from $\dot a$ to $b$ is unique. The latter can fail only if $b$ is
    from the easily described subset of $V_*$ for which the path $\dot a$ to $b$
    is non-unique, but in this case $\eta'_t$ is a sum of copies of $\eta_t$ is
    duplicated on each path, so it is enough to understand $\eta_t$.  To achieve
    the former we can decompose $A$ into subsets.  From the proof of
    Theorem~\ref{thm:nullset} the set of points in $A$ which are equidistant
    from $b$ along paths through distinct boundary points lies on a hyperplane
    orthogonal to the edge between these boundary points.  This is a measure
    zero set so can be deleted without affecting $\eta_t$ (or $\eta'_t$).
    Repeating this for each pair of boundary points we are left with finitely
    many open subsets of $A$, each of which is then a countable union of cells.
    Each cell obtained in this manner has both of our assumed properties, so
    $\eta'_t$ can be written as a countable (and locally finite) sum of measures
    of the type $\eta_t$.
\end{remark}

\begin{lemma}\label{lem:Z_tbydist}
Under the above assumptions, for $t\in[0,1]$,
\begin{equation*}
	Z_{t,b}^{-1}\bigl(\hat\gamma(s)\bigr) = \Bigl\{a\in A: \frac{d(a,\dot a)}{d(\dot a,b)} = \frac{1-s}{1-t} -1\Bigr\},
	\end{equation*}
which is non-empty when $1-\frac{\diam(A)}{d(\dot a,b)}\leq \frac{1-s}{1-t}\leq 1$.  This is illustrated in Figure~\ref{fig:preimofpt}.
\end{lemma}
\begin{proof}
Recall that $\hat\gamma$ has constant speed parametrization, so $x=\hat\gamma(s)$ implies $d(\dot a,x)=sd(\dot a,b)$ and $d(x,b)=(1-s)d(\dot a,b)$.  From~\eqref{eq:Zt} the set $Z_{t,b}^{-1}(x)$ consists of those $a\in A$ so $x\in Z_t(a,b)$, which means $d(a,x)=td(a,b)$ and $d(x,b)=(1-t)d(a,b)$.  However, $\dot a$ is on the geodesic from $a$ to $b$ and the geodesic from $a$ to $x$, so we have both $d(a,b)=d(a,\dot a)+d(\dot a,b)$ and  $d(a,x)=d(a,\dot a)+d(\dot a,x)$. From this
\begin{equation*}
	d(a,\dot a)+ sd(\dot a, b) = d(a,\dot a)+ d(\dot a,x)=d(a,x)=td(a,b)  = td(a,\dot a)+td(\dot a, b)
	\end{equation*}
which may be rearranged to obtain
\begin{equation*}
	d(a,\dot a) = \Bigl( \frac{1-s}{1-t} - 1\Bigr) d(\dot a, b)
	\end{equation*}
and therefore the desired expression for $Z_{t,b}^{-1}(\hat\gamma(s))$.  The condition for the set to be non-empty is a consequence of there being points $a\in A$ with $0\leq d(a,\dot a)\leq \diam (A)$.
\end{proof}

\begin{figure}[ht]
  \centering
  \begin{tikzpicture}[x=0.9*\sqTHREE in,y=.9in]
	\draw[black] (0,0) -- (0,1) -- (1/2,1/2) -- cycle;
	\draw[black, thick] (1/2,1/2) -- (2,1/2);
    \draw[black, dashed] (1/8,1/8) -- (1/8,7/8);
    \draw (1/8,1/8) node[below right]{$Z_{t,b}^{-1}(x)$}; 
    \filldraw (1,1/2) circle (1.5pt) node[above]{$x$}; 
    \draw[draw=none] (0,0) -- (0,1) -- (1,1/2) -- cycle;
    \filldraw (1/2,1/2) circle (1.5pt) node[above]{$\dot a$};
    \filldraw (2,1/2) circle (1.5pt) node[above]{$b$};
    \draw (5/4,1/2) node[below]{$\hat\gamma$};
    \draw (0,1/2) node[left]{$A$};
  \end{tikzpicture}
  \caption{The preimage of $x$ under $Z_{t,b}$ consists of points equidistant from $\dot a$.}
  \label{fig:preimofpt}
\end{figure}
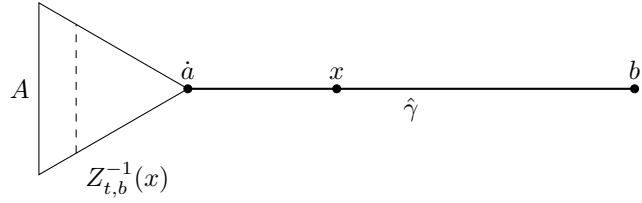  

From Proposition~\ref{dist-bary} the set of points in $A$ at a prescribed distance from $\dot a$ is a level set of the barycentric coordinate corresponding to $\dot a$, see Figure~\ref{fig:preimofpt}. We define an associated projection.

\begin{defn}\label{def:proj}
If $A=\br{w}$ and $\dot a=\br{w\bar i}$, let  $\varphi_{\dot a}(y)=[F_w^{-1}(y)]_i$, so $\varphi_{\dot a}:\br{w}\to[0,1]$ is the projection of $A$ on the scaled barycentric coordinate with $\varphi_{\dot a}(\dot a)=1$ and $\varphi_{\dot a}=0$ at the other boundary points of $A$. Note that $d(a,\dot a)=2^{-|w|}(1-\varphi_{\dot a}(a))$ for $a\in A$.
\end{defn}

In particular, the projection allows us to use the parametrization $H_t$ from Definition~\ref{def:Ht} to give a more convenient version of Lemma~\ref{lem:Z_tbydist} when $t\in[t_1^f,1]$.
\begin{lemma}\label{lem:ZtviaHt}
For $t\in[t_1^f,1]$, we have $Z_{t,b}^{-1}(s)=\varphi_{\dot a}^{-1}\circ H_t^{-1}$.
\end{lemma}
\begin{proof}
Since both $Z_{t,b}$ and $H_t\circ\varphi_{\dot a}$ map $A\to Z_t(A,b)$, are constant on level sets of $\varphi_{\dot a}$ and linear with respect to distance, they are equal.
\end{proof}

These considerations further suggest we consider a pushforward measure under the scaled barycentric projection.
\begin{defn}\label{def:nu}
Let  $\nu_n$ be the pushforward measure $\nu_n(X) = (\varphi_{\br{\bar 0}})_*\mu_n (X)= \mu_n\circ\varphi_{\br{\bar 0}}^{-1}(X)$ on Borel subsets of $[0,1]$.
\end{defn}
As $S_n$ is rotationally symmetric, we could have defined $\nu_n$ using  any boundary point map
$\varphi_{\br{\bar i}}$, and obtained the same measure.  Moreover, the fact that $\varphi_{\dot a}^{-1}
=F_w\circ\varphi_{\br{\bar i}}^{-1}$ implies that $\mu_n\circ \varphi_{\dot a}^{-1}=(n+1)^{-|w|}\nu_n$. 
 It is equally important that $\nu_n$ satisfies a simple self-similarity condition.

\begin{lemma}\label{lem:nunss}
If $\tilde F_i=\varphi\circ F_i\circ \varphi^{-1}$ then 
$\nu_n=\frac1{n+1}\nu_n\circ\tilde F_0^{-1} + \frac n{n+1} \nu_n\circ \tilde F_1^{-1}$. 
\end{lemma}
\begin{proof}
Recall that $\varphi(q_0)=1$ and $\varphi(q_j)=0$, for $j\ne 0$, while from Lemma~\ref{prop:general-IFS} we have $\tilde F_j(x)=\half(x+\varphi(q_j)$.  Thus
$\tilde F_0(x)=\half(x+1)$ and $\tilde F_j(x)=\half x$ if $j\neq 0$.  Proposition \ref{prop:general-pushforward} says that  $\nu_n$ is self-similar under the IFS $\{\tilde F_i\}$ with equal weights, and the result follows from the fact that $n$ of these maps are the same.
\end{proof}

See Figure \ref{fig:nu} for the approximate density of $\nu_2$, where the
weights are $\frac{1}{3}$ and $\frac{2}{3}$.

\begin{figure}[ht]
	\centering
	\includegraphics[width=5.0in]{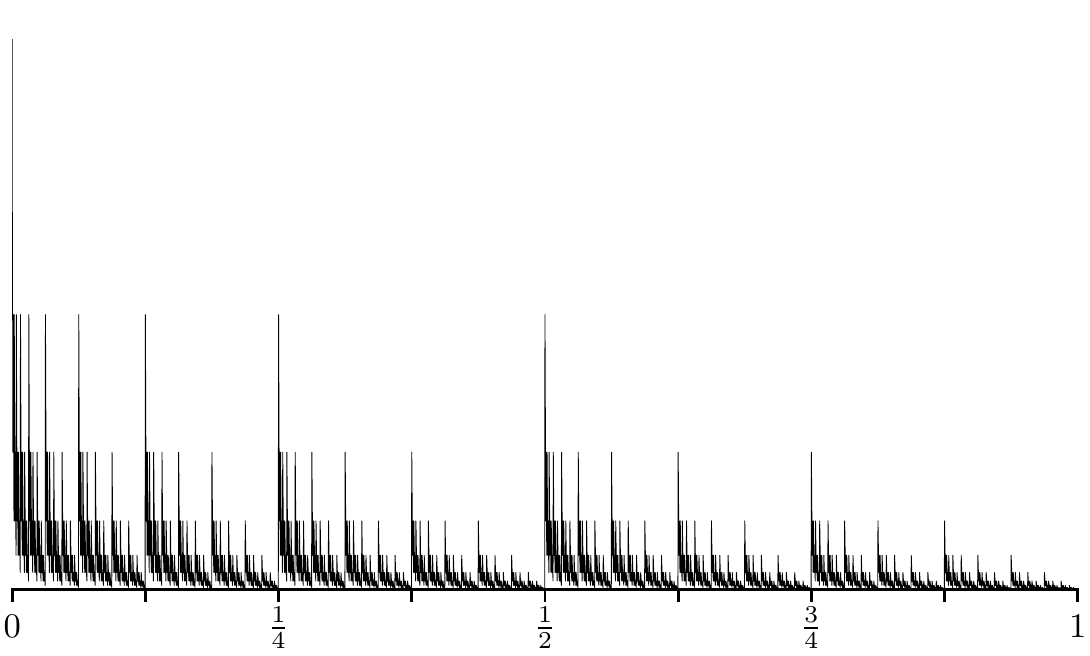}
    \caption{Approximate density of the self-similar measure $\nu_2$.}
    \label{fig:nu}
\end{figure}

Since the IFS $\{\tilde F_0,\tilde F_1\}$ satisfies the open set condition it is fairly elementary to compute the  Hausdorff dimension of $\nu_n$,   for example using the approach in Chapter~5.2 of~\cite{Edgar2}.  One expression for this dimension is $\inf\{\dim_{\text{Hausd}} (E):\nu_n(E)>0\}$.
\begin{prop}\label{prop:HdimZtb}
The Hausdorff dimension of  $\nu_n$ is $\frac{(n+1)\log(n+1)-n\log n}{(n+1)\log 2}$.  In particular it is singular with respect to Lebesgue measure on $[0,1]$.
\end{prop}

With the pushforward measure $\nu_n$ in hand, we can give an elementary and concise description of the common path measure $\eta_t$ using Lemma~\ref{lem:ZtviaHt}; it is the main result of this section.

\begin{thm} \label{thm:eta-nu}
Let $A=\br{w}$ be a cell and $B=\{b\}$ with $b\not\in A$. If $t\in[t_1^f,1]$ then $\eta_t =  (n+1)^{-|w|} \nu_n\circ H_t^{-1}$, so is singular with respect to arc length and has dimension as in Proposition~\ref{prop:HdimZtb}.
\end{thm}
\begin{proof}
We have $ \eta_t= \mu_n\circ Z_{t,b} ^{-1} = \mu_n\circ \varphi_{\dot a}^{-1} \circ H_t^{-1} =  (n+1)^{-|w|} \nu_n \circ H_t^{-1}$.
\end{proof}

It should be remarked that we could have described $\eta_t$ for $t\in[0,1]$ rather than only $t\in[t_1^f,1]$ by using Lemma~\ref{lem:Z_tbydist} instead of  Lemma~\ref{lem:ZtviaHt}, but the notation is considerably less elementary and  the gain is minimal because in this case one can instead compute $\eta_t$ for the largest subcell $A'\subset A$ such that $t>t_1^f$ for $A'$.



\section{Interpolation of measures}
\label{sec:interp-general}

The general interpolation problem involves understanding $\{(a,b):Z_t(a,b)=x\}\subset A\times B$ and its product measure.  We slightly abuse notation by calling this measure $\eta_t$, as we did in the case $B=\{b\}$
\begin{defn}\label{def:etainhigherdim}
Let $A$ and $B$ be sets of nonzero $\mu_n$-measure, and suppose there is a regular common path
$\hat\gamma$ between them. Define a measure $\eta_t$ on $\hat\gamma$ to be the pushforward of $\mu_n\times\mu_n$ on $A\times B$, so that for each $t\in[0,1]$ and Borel set $X$,
\[	\eta_t(X) = (\mu_n\times\mu_n)\circ Z_t^{-1}(X).
\]
\end{defn}

As we did in the case of interpolation between a cell and a point, we take the viewpoint that interpolation between  sets $A$ and $B$ should be understood as a superposition of interpolation between pairs of cells.  This is by no means always possible, but it is possible for a large class of sets; for example, it is true  when $A$ and $B$ are both open.  Using the same considerations made when discussing point to set interpolation, we further note that from the proof of Theorem~\ref{thm:nullset} the product $A\times B$ may be decomposed into a $\mu_n\times\mu_n$-nullset, which is obtained as a finite union of sets of the type in~\eqref{eqn:nullseteqn}, and a countable union $A_j\times B_j$ in which $A_j$ and $B_j$ are disjoint cells joined by a unique common path.  Accordingly, we focus our investigation on $\eta_t$ when $A$ and $B$ are as in Definition~\ref{def:etainhigherdim}. 
 
We conclude with a discussion of interpolation when $\mu_n$ is replaced with an unequally distributed self-similar measure.

\subsection{Cell-to-cell interpolation} \label{subsec:cell-2-cell}

Let $A$ be a $k$-level cell and $B$ an $m$-level cell for which there is a common path $\hat\gamma$ which is the unique geodesic joining the boundary points $\dot a\in A$ and $\dot b\in B$.  From Lemma~\ref{lem:Z_tbydist} we know that $Z_t(a,b)=Z_t(a',b)$ if $\varphi_{\dot a}(a)=\varphi_{\dot a}(a')$ and similarly for the second coordinate using $\varphi_{\dot b}$, so it is natural to write $Z_t^{-1}(x)$ using these barycentric coordinates.  Note that they are scaled differently on $A$ and $B$, as in the following definition.

\begin{defn} \label{def:psi}
Define $\psi_t:[0,1]\times[0,1]\to[0,1]$ by
\[ 
  \psi_t(s,r) = \frac{2^{-k}(1-t)s+2^{-m}t(1-r)}{2^{-k}(1-t)+2^{-m}t}.
\]
\end{defn}

The following result is similar to Lemma~\ref{lem:ZtviaHt} and is illustrated in Figure \ref{fig:cell2cell}.  Recall from Definition~\ref{def:Ht} that $H_t$ parametrizes $Z_t(A,B)$ when the latter is contained in $\hat\gamma$.  

\begin{lemma} \label{commute-lemma}
For all $t\in[t^f_1,t^i_2]$  we have 
%
$	Z_t(a,b) = H_t \circ \psi_t(\varphi_{\dot{a}}(a),\varphi_{\dot{b}}(b)).
$
\end{lemma}
\begin{proof}
Recall that $Z_t(a,b)=x$ means $d(a,x)=td(a,b)$.  Suppose now that $x=H_t\circ\psi_t\circ(\varphi_{\dot{a}}(a),\varphi_{\dot{b}}(b))$.  We establish several points that together show $d(a,x)=td(a,b)$, proving the result.

Recall from Definition~\ref{def:proj} that $d(a,\dot a)=2^{-k}(1-\varphi_{\dot a}(a))$ and $d(\dot b,b)=2^{-m}(1-\varphi_{\dot b}(b))$.  Substituting into $\psi_t$ gives
\begin{equation}\label{eqn:psiteqn1}
	\psi_t \bigl(\varphi_{\dot{a}}(a),\varphi_{\dot{b}}(b)\bigr)
	= \frac{(1-t)(2^{-k}-d(a,\dot a))+td(\dot b,b)}{2^{-k}(1-t)+2^{-m}t}.
	\end{equation}

To proceed we need more information about $H_t$, the parametrization of $Z_t(A,B)$.  Using Lemma~\ref{lem:Z_tbydist} we find that the extreme points of $Z_t(A,B)$ are $x_1$ and $x_2$ satisfying $d(\bar a,x_1)=td(\bar a,\dot b)$ and $d(x_2,\bar b)=(1-t)d(\dot a,\bar b)$, where $\bar a\neq \dot a$ is a boundary point of $A$ and $\bar b\neq \dot b$ is a boundary point of $B$.  Since $d(\bar a,\dot a)=2^{-k}$ and $d(\dot b,\bar b)=2^{-m}$ this yields $d(\bar a,x_1)=t2^{-k} + td(\dot a,\dot b)$ and  $d(x_2,\bar b)=(1-t)2^{-m}+(1-t)d(\dot a,\dot b)$.  Moreover $t\in[t_1^f,t_2^i]$ implies that for any $a\in A$ and $b\in B$ the geodesic from $a$ to $b$ contains the following points in order: $a,\dot a,x_1,x,x_2,\dot b, b$.   We use this and the side lengths of the cells $A$ and $B$ to determine that
\begin{align*}
	d(x_1,x_2) & = d(\bar a,\bar b)- d(\bar a,x_1)-d(x_2,\bar b)\\
			& = 2^{-k}+2^{-m}+d(\dot a,\dot b) - td(\bar a, \dot b) - (1-t)d(\dot a,\bar b)\\
			&=  2^{-k}+2^{-m}+d(\dot a,\dot b) - t2^{-k} -td(\dot a,\dot b) - (1-t)2^{-m} - (1-t)d(\dot a,\dot b)\\
			&= 2^{-k}(1-t) + 2^{-m} t
	\end{align*}
which is the denominator in $\psi_t$.

Now $H_t$ is the linear parametrization of the path from $x_1$ to $x_2$, so $x=H_t(q)$ means $d(x_1,x)=qd(x_1,x_2)$.  Substituting $q=\psi_t \bigl(\varphi_{\dot{a}}(a),\varphi_{\dot{b}}(b)\bigr)$ from~\eqref{eqn:psiteqn1} we have
\begin{equation*}
	d(x_1,x)=  d(x_1,x_2)\psi_t \bigl(\varphi_{\dot{a}}(a),\varphi_{\dot{b}}(b)\bigr)= (1-t)(2^{-k}-d(a,\dot a))+td(\dot b,b).
	\end{equation*}
We can then compute $d(a,x)$ as follows, using $2^{-k}+d(\dot a,x_1)=d(\bar a,x_1)=t2^{-k}+td(\dot a,\dot b)$.
\begin{align*}
	d(a,x)&=d(a,\dot a)+d(\dot a,x_1)+d(x_1,x)\\
	& = d(a,\dot a)+d(\dot a,x_1)+(1-t)(2^{-k}-d(a,\dot a))+td(\dot b,b)\\
	&= td(\dot a,\dot b) +t d(a,\dot a) + td(\dot b,b)=td(a,b)
	\end{align*}
from which $x=Z_t(a,b)$ as required.
\end{proof}

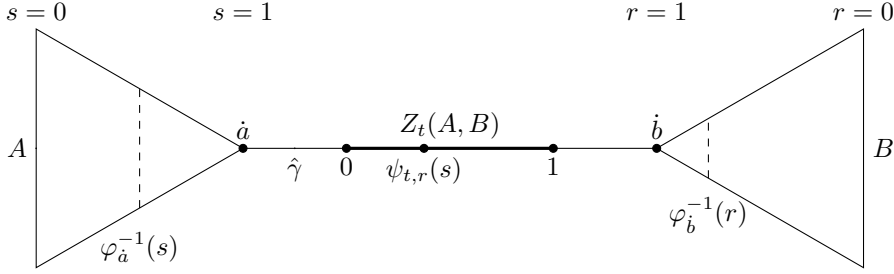
\begin{figure}[ht] 
  \centering
  \begin{tikzpicture}[x=0.625*\sqTHREE in,y=1.25in]
	\draw[black] (0,0) -- (0,1) -- (1,1/2) -- cycle;
	\draw[black] (1,1/2) -- (3,1/2);
    \draw[black] (3,1/2) -- (4,0) -- (4,1) -- cycle;
    \draw[black, dashed] (1/2,1/4) -- (1/2,3/4);
    \draw[black, dashed] (13/4,3/8) -- (13/4,5/8);
    \filldraw (0,1/2) circle (0.1pt) node[left]{$A$};
    \filldraw (0,1) node[above]{$s=0$};
    \filldraw (1,1) node[above]{$s=1$};
    \filldraw (3,1) node[above]{$r=1$};
    \filldraw (4,1) node[above]{$r=0$};
    \filldraw (4,1/2) node[right]{$B$};
    \filldraw (1,1/2) circle (1.5pt) node[above]{$\dot{a}$};
    \filldraw (3,1/2) circle (1.5pt) node[above]{$\dot{b}$};
    \filldraw (1/2,1/2) node[below, yshift=-28]{$\varphi_{\dot{a}}^{-1}(s)$};
    \filldraw (13/4,1/2) node[below, yshift=-15]{$\varphi_{\dot{b}}^{-1}(r)$};
    \draw (5/4,1/2) circle (0.1pt) node[below]{$\hat\gamma$};
    \draw[black, very thick] (3/2,1/2) -- (5/2,1/2);
    \filldraw (3/2,1/2) circle (1.5pt)node[below]{$0$};
    \filldraw (5/2,1/2) circle (1.5pt)node[below]{$1$};
    \draw (2,1/2) node[above]{$Z_t(A,B)$};
    \filldraw (15/8,1/2) circle (1.5pt) node[below]{$\psi_{t,r}(s)$};
  \end{tikzpicture}
  \caption{A schematic of cell-to-cell interpolation on a common path
  $\hat\gamma$.  The function $\psi_{t,r}$ (Definition \ref{def:psi}) describes at a given $t$ 
  where in the interval $Z_t(A,B)$ a point lying on the line $\varphi_{\dot{a}}^{-1}(s)$ is as 
  it is interpolated to a point lying on the line $\varphi_{\dot{b}}^{-1}(r)$.}
\label{fig:cell2cell}
\end{figure}



We can now prove an analogue of 
Theorem \ref{thm:eta-nu} for cell-to-cell interpolation.

\begin{thm} \label{thm:eta-conv}
If $A$ is a $k$-level cell and $B$ an $m$-level cell that are joined by a regular common path $\hat \gamma$ that is the unique geodesic between boundary points $\dot a\in A$ and $\dot b\in B$, then for all $t\in[t^f_1,t^i_2]$ 
\begin{equation*}
  \eta_t =  (n+1)^{-k-m} (\nu_n \times \nu_n)\circ \psi_t^{-1}\circ H^{-1}_t.
  \end{equation*}
\end{thm}
\begin{proof}
This is an immediate consequence of Lemma~\ref{commute-lemma} applied to the definition of $\eta_t$, because the functions $\varphi_{\dot a}^{-1}$ and $\varphi_{\dot b}^{-1}$ may be pulled into the product measure as follows:
\begin{equation*}
	\eta_t= (\mu_n\times\mu_n) \circ Z_t^{-1}= (\mu_n\times\mu_n)\circ ( \varphi_{\dot a},\varphi_{\dot b})^{-1}\circ \psi_t^{-1}\circ H_t^{-1}
	= (\mu_n\circ \varphi_{\dot a}^{-1} \times \mu_n\circ\varphi_{\dot b}^{-1}) \circ \psi_t^{-1}\circ H_t^{-1}
	\end{equation*}
so we can use $\mu_n\circ \varphi_{\dot a}^{-1}=(n+1)^{-k}\nu_n$ and similarly $\mu_n\circ\varphi_{\dot b}^{-1}=(n+1)^{-m}\nu_n$.
\end{proof}

Since it is a product of self-similar measures, the measure $\nu_n\times\nu_n$ is self-similar.  This is recorded in Proposition~\ref{prop:nuproductss} after defining notation for the two-dimensional IFS.  It is illustrated in Figure~\ref{fig:unit-square}.

\begin{defn} \label{def:g}
Let $q_{00}=(0,0), q_{01}=(0,1), q_{10}=(1,0)$, and $q_{11}=(1,1)$.  
For $i,j\in\{0,1\}$, define $G_{ij}:[0,1]^2 \to [0,1]^2$ by \[
  G_{ij}(x) = \frac{1}{2}(x + q_{ij}), 
\] and fix weights $w_{ij}=w_iw_j$, where $w_0=\nnpo$ and $w_1=\onpo$.
\end{defn}

The functions $G_{ij}$ are an IFS generating the unit square 
and are 
related to the functions $\tilde F_i$ by \[
  G_{ij}\mqty(x \\ y) = \mqty(\tilde F_{1-i}(x) \\ \tilde F_{1-j}(y)).
\]

\begin{figure}[ht]
\begin{tikzpicture}[x=1.5in, y=1.5in]
  \draw[black, thick] (0,0) rectangle (1,1);
  \draw (1/2,0) -- (1/2,1);
  \draw (0,1/2) -- (1,1/2);
  \draw (-0.169,1/2) node {\includegraphics[width=1.5in,height=0.5in,
    angle=90,origin=c]{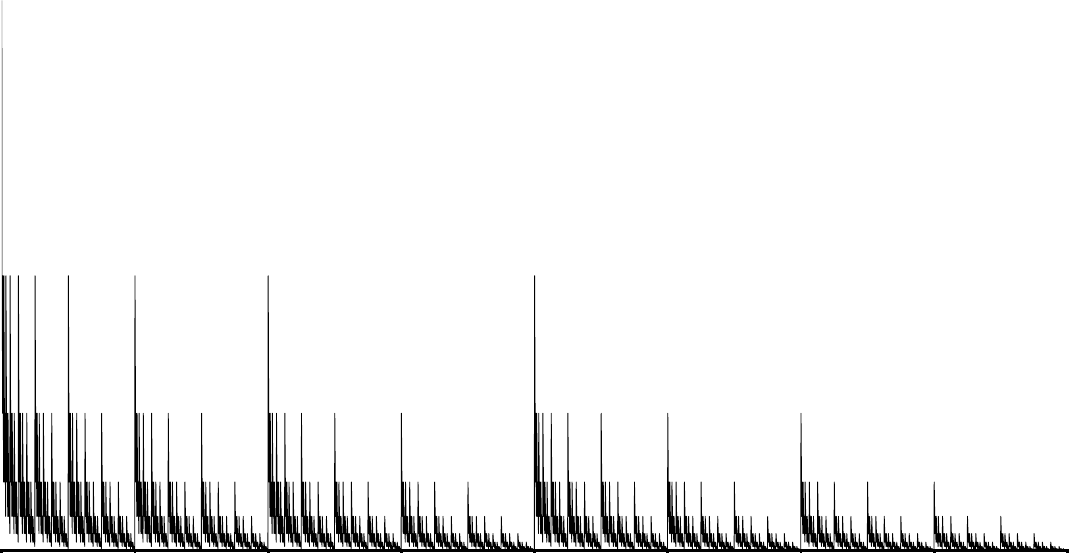}};
  \draw (1/2,-0.169) node {\reflectbox{\includegraphics[width=1.5in,height=0.5in,
  	angle=180,origin=c]{nu_crop}}};
  \draw (1/4,1/4) node{$\frac{n^2}{(n+1)^2}$};
  \draw (3/4,1/4) node{$\frac{n}{(n+1)^2}$};
  \draw (1/4,3/4) node{$\frac{n}{(n+1)^2}$};
  \draw (3/4,3/4) node{$\frac{1}{(n+1)^2}$};
\end{tikzpicture}
\caption{A diagram of the product measure $\nu_n\times\nu_n$.}
\label{fig:unit-square}
\end{figure}

\begin{prop}\label{prop:nuproductss} The measure $\nu_n\times \nu_n$ satisfies the self-similar relation
 \[
	\nu_n\times\nu_n = \sum_{i,j} w_{ij}(\nu_n\times\nu_n)\circ G_{ij}^{-1}.
\]
\end{prop}
\begin{proof}
We compute from the self-similarity of $\nu_n$ that
\begin{align*}
  \nu_n \times \nu_n 
   &= \Bigl(\sum_i w_i  \nu_n\circ \tilde F_{1-i}^{-1}\Bigr)\Bigl(  \sum_j w_j \nu_n\circ \tilde F_{1-j}^{-1}\Bigr) \\
  &= \sum_{i,j}w_iw_j (\nu_n\times\nu_n)\circ (\tilde F_{1-i}^{-1}\times \tilde F_{1-j}^{-1})\\
  &= \sum_{i,j}w_{ij} (\nu_n \times \nu_n)\circ G_{ij}^{-1}.\qedhere
\end{align*}
\end{proof}

Theorem~\ref{thm:eta-conv} establishes that $\eta_t$ depends only on the linear parametrization $H_t$ of $Z_t(A,B)$ and the pushforward measure
\begin{equation}\label{tildenu}
	\tilde \nu_n^t =(\nu_n\times\nu_n)\circ\psi_t^{-1}.
	\end{equation}
This measure has a simple geometric meaning. Observe that $\psi_t$ is a scaled projection from the unit square to the unit interval along lines
of slope $2^{k-m}(\frac{t}{1-t})$.  The corresponding pushforward is then a generalization of a convolution; the usual convolution $\nu_n*\nu_n$ occurs when the lines have slope $-1$.  From Proposition~\ref{prop:general-pushforward} we also find that $\tilde \nu_n^t $ is self-similar.

\begin{thm}\label{thm:eta-general}
For $t\in[0,1]$ let $\tilde G_{ij}=\psi_t\circ G_{ij}\circ\psi_t^{-1}:[0,1]\to[0,1]$.  Then 
\[
 \tilde \nu_n^t = \sum_{i,j}w_{ij} \tilde \nu_n^t\circ \tilde G_{ij}^{-1}. \]
\end{thm}

The maps in the IFS $\{\tilde G_{ij}\}$ take $[0,1]$ to overlapping segments in $[0,1]$, with overlaps that depend on $t$.  Figure \ref{fig:g} shows these overlapping segments, along with their corresponding weights, for one choice of $t$, and Figure \ref{fig:nu-star-nu} shows the approximate densities of $\tilde\nu_n^t$ for 
several $t$ values.

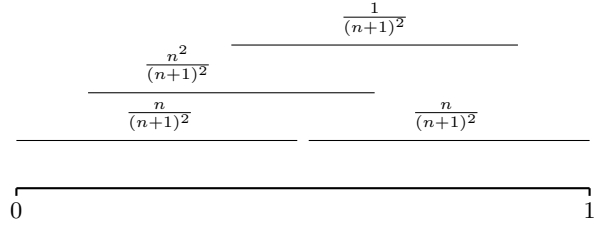
\begin{figure}[ht]\label{self-sim-fig}
\begin{tikzpicture}[x=3in,y=2in]
  \draw[black, thick] (0,0) -- (1,0);
  \draw (0,0) node[below, yshift=-2]{\small $0$};
  \draw (1,0) node[below, yshift=-2]{\small $1$};
  \draw[black] (0,1/8) -- (0.49, 1/8);
  \draw (1/4,1/8) node[above]{\small $\frac{n}{(n+1)^2}$};
  \draw[black] (0.51,1/8) -- (1, 1/8);
  \draw (3/4,1/8) node[above]{\small $\frac{n}{(n+1)^2}$};
  \draw[black] (1/8,1/4) -- (5/8,1/4);
  \draw (9/32,1/4) node[above]{\small $\frac{n^2}{(n+1)^2}$};
  \draw[black] (3/8,3/8) -- (7/8,3/8);
  \draw (5/8,3/8) node[above]{\small $\frac{1}{(n+1)^2}$};
  \draw [black, thick] (0,0) -- (0,-0.02);
  \draw [black, thick] (1,0) -- (1,-0.02);
\end{tikzpicture}
\caption{Distribution of self-similar weights of $\tilde\nu_n^t$ for some $t$.} 
\label{fig:g}
\end{figure}

\begin{figure}[ht]\label{conv-densities}
	\centering
	\begin{subfigure}[b]{0.45\textwidth}
        \centering
		\includegraphics[width=2in]{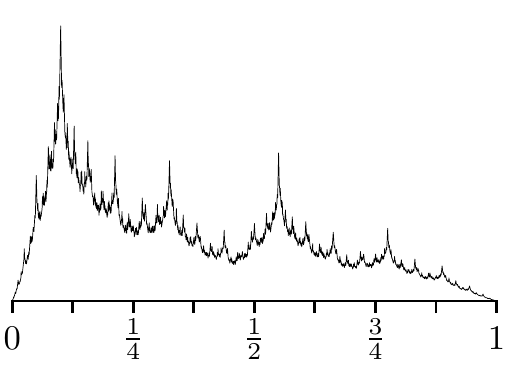}	
        \caption{$t=0.10$}
	\end{subfigure}
	\begin{subfigure}[b]{0.45\textwidth}
        \centering
		\includegraphics[width=2in]{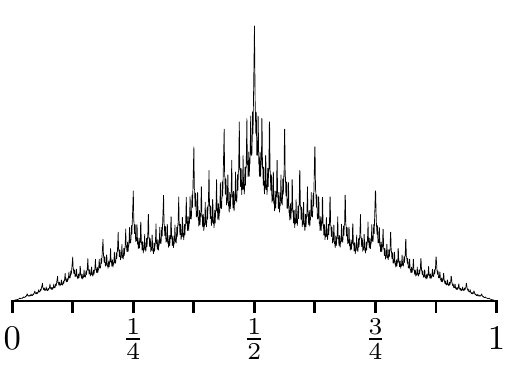}	
        \caption{$t=0.50$}
	\end{subfigure}
	\begin{subfigure}[b]{0.45\textwidth}
        \centering
		\includegraphics[width=2in]{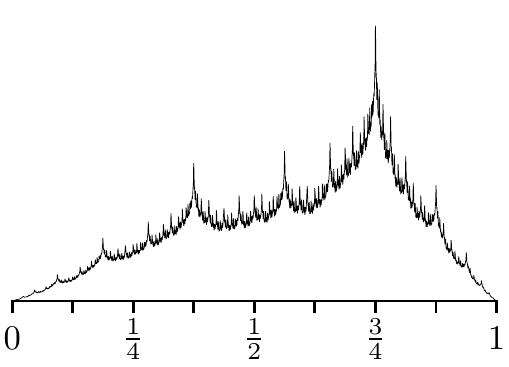}	
        \caption{$t=0.75$}
	\end{subfigure}
	\begin{subfigure}[b]{0.45\textwidth}
        \centering
		\includegraphics[width=2in]{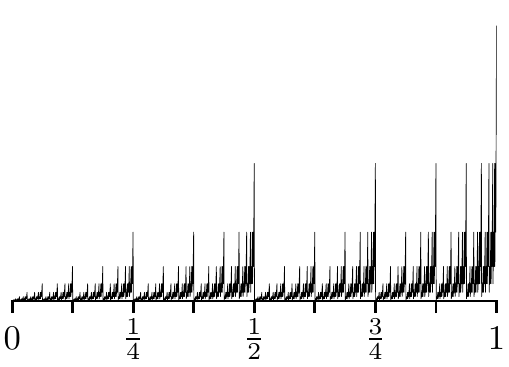}	
        \caption{$t=1.00$}
	\end{subfigure}
    \caption{Approximate densities of the convolution measure $\tilde\nu_n^t$ at 
    various values of $t$.}
    \label{fig:nu-star-nu}
\end{figure}

It is generally difficult to compute the dimensions of measures from overlapping IFS, but we may deduce some results from the Marstrand projection theorem~\cite{marstrand1954}.  First note that the dimension of $\nu_n\times\nu_n$ is twice that of $\nu_n$, and is given by the formula
\begin{equation} \label{lower-hausdorff-gen}
 2\frac{(n+1)\log(n+1)-n\log n}{(n+1)\log 2}.
 \end{equation}
This expression is decreasing with limit zero as $n$ increases. In particular, it is less that $1$ for $n\geq9$ and greater than $1$ for $2\leq n\leq 8$, from which we deduce the following using Theorems~6.1 and~6.3 in~\cite{hu1994fractal}.
\begin{thm}\label{thm:celltocellabscty}
If $2\leq n\leq 8$ then for almost all $t\in[0,1]$ the measure $\tilde\nu_n^t$ is absolutely continuous with respect to Lebesgue measure on $[0,1]$.  For $n\geq9$, it is  singular with respect to Lebesgue measure, and in fact has lower Hausdorff dimension given by~\eqref{lower-hausdorff-gen}.
\end{thm}
We note that recent results of Shmerkin and Solomyak~\cite{shmerkin2016} show the set of exceptional $t$ in this theorem is not just zero measure but zero Hausdorff dimension.

\subsection{Alternate weightings of self-similar measures on the gasket} \label{subsec:alt-weight-meas}

We can generalize our previous results to self-similar measures other than the standard measure on $S_n$. 
Consider a self-similar measure $\mu_n'$ on $S_n$ given by weights $\{\mu_n^i\}_{i=0}^n$.  In this
case, the pushforward measure $\nu_n' = \varphi_*\mu_n'$ is self-similar, but has weights dependent on 
the reference point of the projection $\varphi$.  In particular, if we consider the projection with 
respect to a vertex $\br{w\bar i}$, the self-similarity relation is given by: \[
    \nu_n'(X) = \mu_n^i\nu_n'\circ \tilde F_1^{-1}(X) 
    + \qty(\sum_{j\ne i}\mu_n^j)\nu_n'\circ \tilde F_0^{-1}(X).
\] This follows from Lemma~\ref{prop:general-IFS} and Proposition~\ref{prop:general-pushforward}. 
Self-similarity also carries over to $\tilde\nu_n^t$, but
the self-similarity weights depend on orientations of both the starting and ending cells with respect
to the common path; if $\br{w\bar i}$ is the entry point and $\br{v\bar j}$ the exit point of a 
common path, then $\tilde\nu_n^t$ has self-similarity relations as in Theorem~\ref{thm:eta-general}, but with $w_{00}=\mu_n^i\mu_n^j$, $w_{01}=\mu_n^i(\sum_{k\ne j}\mu_n^k)$, $w_{10}
=(\sum_{k\ne i}\mu_n^k)\mu_n^j$, and $w_{11}=(\sum_{k\ne i}\mu_n^k)(\sum_{k\ne j}\mu_n^k)$.


\section{An Interpolation Inequality} \label{sec:inequality}
The model for an inequality involving the interpolant set $Z_t$ is the classical Brunn--Minkowski inequality, which says that for sets in $\mathbb{R}^n$ the Euclidean volume $|\cdot|$ satisfies $\bigl|Z_t(A,B)|^{1/n}\ge (1-t)|A|^{1/n}+t|B|^{1/n}$.  We have already noted that this inequality cannot be valid for $\mu_n$ because $Z_t(A,B)$ for $t\in (0,1)$ has Hausdorff dimension at most $1$, and thus $\mu_n$-measure zero.  When seeking alternative inequalities it is not entirely clear which measures to use: $\mu_n$ is natural for $A$ and $B$, and is equivalent to $\nu_n$ for the barycentric projection of these sets, at least under the conditions considered in the previous sections, but $\nu_n$ is not natural for $Z_t(A,B)$ because it is defined on $[0,1]$, not on the common path.   Since $\cup_{t\in(0,1)}Z_t$ is at most one-dimensional (from Proposition~\ref{prop:onedim}), the geometrically defined natural measures to consider would seem to be Hausdorff measures of dimension at most one. In light of the work done in the previous sections, natural choices of dimension are that of $\nu_n$, that of $\nu_n\times\nu_n$, and $1$. The first is likely to be uninteresting, because if $\nu_n(A)\nu_n(B)>0$ then $Z_t(A,B)$ has dimension at least $\min\{1,2\dim(\nu_n)\}$  as seen in Theorem~\ref{thm:celltocellabscty}. But the others also present some issues with optimality, validity or both: for example, if $A$ and $B$ are connected then $Z_t$ contains an interval and therefore has infinite measure in dimensions less than one, so any inequality for dimension less than one will be trivially true on connected sets, while Theorem~\ref{thm:celltocellabscty} ensures that for $n\geq9$ one could have $\nu_n(A)=\nu_n(B)=1$ and yet $Z_t(A,B)$ has dimension less than one, so no inequality for Hausdorff $1$-measure can be true for general $A$ and $B$ in this case.

In this section we derive an inequality in the case where $A$ and $B$ are connected sets, so $Z_t(A,B)$ contains an interval and therefore the correct measure to use for $Z_t$ is the one-dimensional Hausdorff measure $\mathcal{H}^1$.  There is an easy bound if $A$ and $B$ are cells.

\begin{prop}\label{prop:ineqcellcase}
Suppose $A=\br{v}$ and $B=\br{w}$ are disjoint cells. Then we have the sharp inequality
\begin{equation*}
	\mathcal{H}^1(Z_t(A,B))\geq  (1-t) \mu_n(A)^{\log 2/\log(n+1)} + t \mu_n(B)^{\log 2/\log(n+1)}.
	\end{equation*}
\end{prop}
\begin{proof}
Take $a\in A$ and $b\in B$ so that $d(a,b)$ is maximal.  The geodesic from $a$ to $b$ passes through boundary points $\dot a\in A$ and $\dot b\in B$, with $d(a,\dot a)= 2^{-|v|}$ and $d(b,\dot b)= 2^{-|w|}$.  Then $Z_t(A,B)$ contains an interval along this geodesic, and it is easy to compute a lower bound for its length, which gives
\begin{equation*}
	\mathcal{H}_1(Z_t(A,B))\geq (1-t)2^{-|v|}+t2^{-|w|}.
	\end{equation*}
However $\mu_n(A)=(n+1)^{-|v|}$ and $\mu_n(B)= (n+1)^{-|w|}$, from which the assertion is immediate.  Sharpness occurs when $Z_t$ is equal to this interval, which is true provided $t_1^f<t<t_2^i$; this can be arranged by suitably choosing $A$ and $B$.
\end{proof}

If $A$ and $B$ are connected but are not cells then we can take minimal cells $\br{v}\supset A$ and $\br{w}\supset B$. Provided $\br{v}$ and $\br{w}$ are disjoint and joined by a common path from $\dot a$ to $\dot b$ our reasoning from the the proof of Proposition~\ref{prop:ineqcellcase} is still useful, but the lower bound for the $\mathcal{H}_1$ measure must now also involve the sizes of the intervals $\varphi_{\dot a}(A)$ and $\varphi_{\dot b}(B)$ obtained by barycentric projection. (Note that these are intervals because $A$ and $B$ are connected.)  Indeed, the geodesic between the maximally separated points $a\in A$ and $b\in B$ begins with a path of length at least $2^{-|v|}\H_1(\varphi_{\dot a}(A))$ in $A$ and ends with one of length at least $2^{-|w|}\H_1(\varphi(B))$ in $B$, so that
\begin{equation}\label{eqn:lowerboundonH1Zt}
	\mathcal{H}^1(Z_t(A,B))\geq (1-t)2^{-|v|}\H_1(\varphi_{\dot a}(A))+t2^{-|w|}\H_1(\varphi_{\dot b}(B))
\end{equation}
and in order to proceed we must bound  $\H_1(\varphi_{\dot a}(A))$ from below using $\mu_n(F_v^{-1}A)$. It is obvious that $\mu_n(F_v^{-1}(A))\leq \nu_n(\varphi_{\dot a}(A))$.   To compare  $\H_1(\varphi_{\dot a}(A))$ and $\nu_n(\varphi_{\dot a}(A))$  we establish some lemmas; the conclusion of our reasoning regarding a lower bound for $\H_1(Z_t(A,B))$ when $A$ and $B$ are connected is in Theorem~\ref{thm:inequality}.

\begin{lemma} \label{lem:fat-nu}
	If $[a,a+x]\subseteq[0,1]$ then $\nu_n([0,x])\geq \nu_n([a,a+x])\geq \nu_n([1-x,1])$.
\end{lemma}
\begin{proof}
As $\nu_n$ is non-atomic, $\nu_n([a,a+x])$ is continuous in $a$ and $x$, and it
suffices to consider dyadic rationals of arbitrary scale $m$, so $a=\sum_{i=1}^m a_i2^{-i}$ and $x=\sum_{i=1}^m x_i 2^{-i}$.

Observe that  we can assume $x\leq a$ and $a+x\leq 1-x$, because if the intervals intersect then it suffices to prove the inequality for the complement of the intersection (for example, if $x>a$ the first inequality may be proved by showing $\nu_n([0,a])\geq \nu_n([x,a+x])$ because then $\nu_n([0,x])=\nu_n([0,a])+\nu_n([a,x])\geq \nu_n([a,x])+\nu_n([x,a+x])=\nu_n([a,a+x])$). Note in particular that this assumption provides $x\leq\frac12$.

We induct on $m$, with the easily verifiable base case $m=1$.  Supposing it is true to scale $m-1$, take $a$ and $x$ at dyadic scale $m$ and use the  self-similarity of $\nu_n$ from Lemma~\ref{lem:nunss}. If both $a$ and $a+x$ are in $[0,\frac12]$ then the scaling map is $F_0^{-1}(y)=2y$ and thus $\nu_n([a,a+x])=\frac n{n+1} \nu_n([2a,2a+2x)])$. Both $2a$ and $2a+2x$ are dyadic of scale $m-1$, so that $\nu_n([0,2x])\geq \nu_n([2a,2a+2x]) \geq \nu_n([1-2x,1])$ from the inductive assumption. We can then use the self-similarity a second time, in the reverse direction, to obtain the desired inequality. (This latter uses $x\leq\frac12$.)  The proof if both $a$ and $a+x$ are in $[\frac12,1]$ follows the same reasoning but uses $F_1^{-1}(y)=2y-1$ on $y\in[\half,1]$.

For the remaining case we have $x\leq a<\frac12<a+x\leq 1-x$, so we separate at
    $\frac12$ and use the self-similarity to write
\begin{equation}\label{eq:boundsfornunofinterval}
	\nu_n([a,a+x])
	=\nu_n\qty(\qty[a,\frac12])+\nu_n\qty(\qty[\frac12,a+x])
	=\frac n{n+1}\nu_n([2a,1])  + \frac1{n+1}\nu_n([0,2a+2x-1]).
	\end{equation}
Since $2a$ and $2a+2x-1$ are dyadic of scale $m-1$ we can apply the inductive assumption to obtain $\nu_n([2a,1])\leq \nu_n([2a+2x-1,2x])$ and $\nu_n([0,2a+2x-1])\geq \nu_n([1-2x ,2a])$.   Thus
\begin{align*}
	\nu_n([0,x])
	 = \frac n{n+1} \nu_n([0,2x])
	 &= \frac n{n+1} \Bigl( \nu_n([2a+2x-1,2x])+ \nu_n([0,2a+2x-1]) \Bigr)\\
	&\geq \frac n{n+1}\nu_n([2a,1])  + \frac1{n+1}\nu_n([0,2a+2x-1]) \\
	&\geq \frac 1{n+1} \Bigl( \nu_n([2a,1]) +  \nu_n([1-2x ,2a]) \Bigr)\\
	&= \frac1{n+1} \nu_n([1-2x,1]) =\nu_n([1-x,1]),
	\end{align*}
where the beginning and end inequalities again use the self-similarity in reverse, which uses the earlier established fact that $x\leq\frac12$. Comparing the middle term to~\eqref{eq:boundsfornunofinterval} establishes the desired inequality.
\end{proof}

Having determined that $\nu_n([a,a+x])\leq \nu_n([0,x])$, we next look for a minimal concave bounding function having the form found in the classical Brunn-Minkowski inequality.  The fact that the following function bounds $\nu_n([0,x])$ is proved in Corollary~\ref{g-ineq} and illustrated in Figure~\ref{fig:nu-cuml}.  Note that Figure~\ref{fig:nu-cuml} makes it clear this is not the minimal concave bounding function, but only the minimal one having the classical Brunn-Minkowski form.
\begin{defn} \label{def: Phi}
	Let $d_n=\frac{\log2}{\log\frac{n+1}n}$, and   $\Phi_n(x)=\qty(1-(1-x)^{d_n})^{1/d_n}$.
\end{defn}

\begin{lemma} \label{lem: Phi-ineq}
\begin{gather}
	n\Phi_n(2x) \leq (n+1) \Phi_n(x) \text{ if $x\in[0,1/2]$,} \label{eq:Phieq1}\\
	\Phi_n(2x-1)\leq (n+1)\Phi_n (x)  -n  \text{ if $x\in[1/2,1]$.} \label{eq:Phieq2}
	\end{gather}
\end{lemma}
\begin{proof}
Dividing both sides of~\eqref{eq:Phieq1} and taking the $d_n$ power we find it is equivalent to
\[    A_1(x)= 1-(1-2x)^{d_n}\leq  2\bigl(1-\qty(1-x)^{d_n}\bigr)=A_2(x).\]
We have $A_1(0)=0=A_2(0)$. Moreover $A_1'(x)=2d_n(1-2x)^{d_n-1}\leq 2d_n(1-x)^{d_n-1}=A_2'(x)$ because  $0\leq1-2x\leq 1-x$ and $d_n-1\ge 0$. The inequality~\eqref{eq:Phieq1} follows.

The inequality~\eqref{eq:Phieq2} is equivalent to
\begin{equation*}
	A_3(1-x) = (1-(2(1-x))^{d_n})^{1/d_n} \leq (n+1)(1-(1-x)^{d_n})^{1/d_n}-n =A_4(1-x)
	\end{equation*}
for $y=1-x\in[0,1/2]$. We have $A_3(0)=1=A_4(0)$ and compare derivatives as follows:
\begin{equation*}
	A_3'(y) = ( 1-(2y)^{d_n})^{\frac1{d_n}-1} 2^{d_n}y^{d_n-1}
	\leq (n+1) (1-y^{d_n})^{\frac1{d_n}-1} y^{d_n-1}
	=A_4'(y)
	\end{equation*}
because $2^{d_n}\leq(n+1)$ and $d_n\geq1$ gives both $0\leq 1-y^{d_n}\leq 1-(2y)^{d_n}$ on $[0,1/2]$ and $\frac1{d_n}-1\leq 0$. The former    is easily checked using the fact that $2^{-d_n}(n+1)$ is decreasing in $n$ and equal to $1$ when $n=1$.  Thus $A_3(y)\leq A_4(y)$ on $[0,1/2]$ and this establishes~\eqref{eq:Phieq2}.
\end{proof}

\begin{cor} \label{g-ineq}
$\nu_n([0,x]) \le \Phi_n(x) = \qty(1-(1-x)^{d_n})^{1/d_n}$ on $[0,1]$.
\end{cor}

\begin{proof}
As in the proof of Lemma~\ref{lem:fat-nu} it is sufficient (by continuity of the functions) to prove this for dyadic rational $x$, which we do by induction on the degree $m$ of the dyadic rational. The base case is $m=0$ where the equalities $\nu_n(\{0\})=0=\Phi_n(0)$ and $\nu_n([0,1])=1=\Phi_n(1)$ are immediate. If $x=k2^{-m}$ is a dyadic rational we use the self-similarity of $\nu_n$ from Lemma~\ref{lem:nunss}, then the fact that $2x$ and $2x-1$ are dyadic rationals of lower degree so satisfy the inequality by induction, and finally Lemma~\ref{lem: Phi-ineq}, to compute
\begin{equation*}	
	\nu_n([0,x])= \begin{cases}
		\frac n{n+1}\nu_n([0,2x]) \leq\frac n{n+1} \Phi_n(2x)\leq \Phi_n(x) &\text{ if $x\leq\frac12$, }\\
		\frac n{n+1}+ \frac1{n+1}\nu_n([0,2x-1]) \leq \frac n{n+1}+ \frac1{n+1}\Phi_n(2x-1)\leq  \Phi_n(x)  &\text{ if $x>\frac12$. }
		\end{cases}\qedhere
	\end{equation*}
\end{proof}

\begin{figure}[ht]
	\centering
	\includegraphics[width=2.5in]{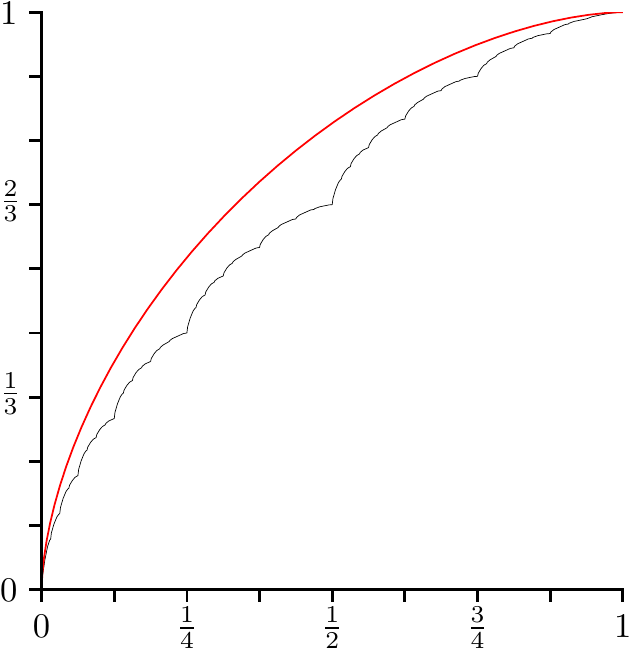}
    \caption{The cumulative measure $\nu_2([0,x])$ (in black) is bounded by $\Phi_2(x)$ (in red).}
    \label{fig:nu-cuml}
\end{figure}

\begin{thm}\label{thm:inequality}
Let $A,B\subset S_n$ be connected sets contained in minimal disjoint cells, $A\subset\br{v}$, $B\subset\br{w}$, and suppose there is a common path between boundary points $\dot a\in\br{v}$ and $\dot b\in\br{w}$.  For all $t\in(0,1)$,
\begin{equation*}
	1- \bigl(1-\H_1(Z_t(A,B))\bigr)^{d_n}
	\geq (1-t)\mu_n(A)^{\log 2/\log (n+1)} + t\mu_n(B)^{\log 2/\log(n+1)}.
	\end{equation*}
\end{thm}
\begin{proof}
The function $\Phi_n^{d_n}$ is concave, so applying it to both sides of~\eqref{eqn:lowerboundonH1Zt} gives
\begin{equation*}
	\Phi_n^{d_n}\bigl( \mathcal{H}_1(Z_t(A,B)) \bigr)
	\geq (1-t) \Phi_n^{d_n}\bigl( 2^{-|v|}\mathcal{H}_1(\varphi_{\dot a} (A))\bigr) + 
		t \Phi_n^{d_n}\bigl( 2^{-|w|} \mathcal{H}_1(\varphi_{\dot b} (B))\bigr).
	\end{equation*}
However we saw in Corollary~\ref{g-ineq} that $\Phi_n(x)$ bounds $\nu_n([0,x])$ and in Lemma~\ref{lem:fat-nu} that $\nu_n([0,x])$ bounds $\nu_n$ of any interval of this length. The latter bound applies to $\varphi_{\dot a}(A)$, which is an interval by the connectedness of $A$.  Also using that $\nu_n([0,2^{-m}x])=\bigl(\frac n{n+1}\bigr)^{m}\nu_n([0,x])$ from the self-similarity of $\nu_n$, we have
\begin{equation*}
	\Phi_n\bigl(2^{-|v|}\mathcal{H}_1(\varphi_{\dot a} (A))\bigr)
	\geq \nu_n\bigl(\bigl[0,2^{-|v|}\mathcal{H}_1(\varphi_{\dot a} (A))\bigr]\bigr)
	\geq \Bigl(\frac n{n+1}\Bigr)^{|v|} \nu_n( \varphi_{\dot a} (A)).
	\end{equation*}
But we also know $\nu_n( \varphi_{\dot a} (A))\geq (n+1)^{|v|} \mu_n(A) =\mu_n(A)/\mu_n(\br{v})$ because the discussion following Definition~\ref{def:nu} showed $\nu_n=(n+1)^{|v|}\mu_n\circ \varphi_{\dot a}^{-1}$.  Using this and the definition of $d_n$ we obtain
\begin{equation*}
	\Phi_n^{d_n}\bigl(2^{-|v|}\mathcal{H}_1(\varphi_{\dot a} (A))\bigr)
	\geq  \Bigl(\frac n{n+1}\Bigr)^{|v|d_n }\Bigl(\frac{\mu_n(A)}{\mu_n(\br{v})}\Bigr)^{d_n}
	=2^{-|v|} \Bigl(\frac{\mu_n(A)}{\mu_n(\br{v})}\Bigr)^{d_n}.
	\end{equation*}
 A similar bound applies for $B$, so we have
\begin{align*}
	\Phi_n^{d_n}\bigl( \mathcal{H}_1(Z_t(A,B)) \bigr)
	&\geq (1-t) 2^{-|v|}\Bigl( \frac{\mu_n(A)}{\mu_n(\br{v})}\Bigr)^{d_n} + 
		t2^{-|w|} \Bigl( \frac{\mu_n(B)}{\mu_n(\br{w})}\Bigr)^{d_n}\\
	&= (1-t) \mu(\br{v})^{d_n'-d_n} \mu_n(A)^{d_n} + 
		t  \mu_n(\br{w})^{d_n'-d_n} \mu_n(B)^{d_n}
	\end{align*}
with $d_n'=\frac{\log 2}{\log (n+1)}$. The fact that $d_n'-d_n=\frac{\log2\log n}{\log (\frac{n+1}n)\log(n+1)}>0$ and both $\mu_n(\br{v})\geq\mu_n(A)$ and  $\mu_n(\br{w})\geq\mu_n(B)$ then gives the result.
\end{proof}

\section{Acknowledgements}
The authors are grateful for constructive discussions with Malcolm Gabbard,
Gamal Mograby, Sweta Pandey, and Alexander Teplyaev.  This project was supported
by the REU site grant NSF DMS~1659643.

  
\bibliography{refs}

\end{document}